\documentclass[11pt]{article}
\usepackage[margin =1in]{geometry}
\usepackage{amssymb,amsthm,amsmath,mathrsfs}
\usepackage{enumerate}
\usepackage{hyperref}
\usepackage{cite}
\usepackage[capitalize]{cleveref}
\usepackage{enumitem}
\usepackage{color}
\usepackage{cancel}
\usepackage{tikz}
\usetikzlibrary{patterns}
\usepackage{pgf}
\usepackage{svg}
\usepackage{pgfplots}
\usepackage{import}

\usepackage[utf8]{inputenc} 
\usepackage[T1]{fontenc}    
\usepackage{hyperref}       
\usepackage{url}            
\usepackage{booktabs}       
\usepackage{amsfonts}       
\usepackage{nicefrac}       
\usepackage{microtype}
\usepackage{multirow}
\usepackage{xfrac}
\usepackage{todonotes}
\usepackage{titlesec}
\usepackage{caption}

\usepackage{multirow}

\usepackage{caption}
\usepackage{float}

\pgfplotsset{compat=1.18}

\titleformat{\subsubsection}[runin]
{\normalfont\normalsize\bfseries}{\thesubsubsection}{1em}{}

\usepackage{graphicx}

\usepackage{appendix}
\usepackage{algorithm}
\usepackage{algpseudocode}
\numberwithin{equation}{section}

\newtheorem{thm}{Theorem}[section]
\newtheorem{theorem}[thm]{Theorem}
\newtheorem{proposition}[thm]{Proposition}

\newtheorem{lemma}[thm]{Lemma}

\newtheorem{corollary}[thm]{Corollary}

\crefname{corollary}{Corollary}{Corollaries}
\crefname{algorithm}{Algorithm}{Algorithms}
\crefname{figure}{Figure}{Figures}

\newcommand{\C}{\mathcal{C}}
\newcommand{\F}{\mathcal{F}}

\newcommand{\Tr}{\mathrm{Tr}}

\newcommand{\xStar}{x_\star}

\newcommand{\diam}{\mathrm{diam}}
\newcommand{\unitv}[2]{\frac{v_{#1}}{{#2}\|v_{#1} \|}}

\newcommand{\unit}[1]{\frac{{#1}}{\|{#1}\|}}

\newcommand{\ob}{\frac{1}{\beta}}
\newcommand{\oa}{\frac{1}{\alpha}}
\newcommand{\og}{\frac{1}{\gamma}}
\newcommand{\K}{\mathcal{K}}
\newcommand{\I}{\mathcal{I}}

\newcommand{\R}{\mathbb{R}}
\newcommand{\interior}{\mathrm{int\ }}
\newcommand{\Interior}[1]{\mathrm{int}_{#1}}

\newcommand{\half}{\frac{1}{2}}

\newcommand{\epi}{\mathrm{epi}}
\newcommand{\obsq}{\frac{1}{\beta^2}}

\newcommand{\argmin}{\operatornamewithlimits{argmin}}

\newcommand{\fmin}{f_\mathrm{min}}

\newcommand{\SSetInterp}{( \{ (z_i,v_i)\}_{i \in \I}, \{(x_k, \delta_k)\}_{k \in \K} )}

\newcommand{\SSpecialInterior}{\left(\{ (z_i,-g_i)\}_{i\in \I}, \ \{(x_0,0), (x_\star, \delta)\} \right)}
\newcommand{\SSpecialExterior}{\left(\{ (z_i,-g_i)\}_{i\in \I} \cup \{(x_\star, -g_\star)\}, \ \{(x_0,0)\} \right)}

\newcommand{\IStar}{\I_\star}
\newcommand{\KStar}{\K_\star}

\newcommand{\SExternal}{S_{g_\star \neq 0}}
\newcommand{\SInternal}{S_{g_\star = 0}^\delta}

\newcommand{\bdry}{\mathrm{bdry\ }}

\newcommand{\conditionSmSCSet}[1]{Q^{i,j}_{#1}}
\newcommand{\conditionSmSCFunc}[1]{\hat{Q}^{i,j}_{#1}}

\newcommand{\rescaledArgs}{\frac{x}{\eta}, \eta g, f}
\newcommand{\standardArgs}{x, g, f}

\newcommand{\gramMatrix}{\Lambda}

\newcommand{\om}{\frac{1}{\mu}}
\newcommand{\ol}{\frac{1}{L}}

\newcommand{\pFWSubproblem}{\hat{p}_\mathrm{FW}}

\newcommand{\ballUB}[1]{\mathfrak{b}_{#1}}
\newcommand{\quadraticUB}[1]{\mathfrak{q}_{#1}}

\newcommand{\el}{l}

\newcommand{\pEpiSmooth}{p_{\mathrm{ES}}}

\newcommand{\sqp}[1]{\sqrt{{#1}+1}}

\newcommand{\setInterpArgs}{(\{((x_i, f_i), (g_i, -1)) \}_{i \in \I}, \emptyset )}
\newcommand{\setInterpRescaledArgs}{(\{((
    \frac{x_i}{\eta}, f_i), (\eta g_i, -1)) \}_{i \in \I}, \emptyset )}

\newcommand{\etaIdx}[1]{{#1}^{(\eta)}}

\newcommand{\edit}[1]{#1}

\newcommand{\myqed}{} 
\newcommand{\branchEpiCoord}{116} 
\newcommand{\branchTriangCoord}{32} 

\theoremstyle{definition}
\newtheorem{remark}{Remark}

\theoremstyle{plain}

\theoremstyle{definition}
\newtheorem{definition}[thm]{Definition}
\allowdisplaybreaks
\begin{document}
    
    \title{Performance Estimation for Smooth and Strongly Convex Sets}

    \author{Alan Luner\footnote{Johns Hopkins University, Department of Applied Mathematics and Statistics, \url{aluner1@jhu.edu}} \qquad Benjamin Grimmer\footnote{Johns Hopkins University, Department of Applied Mathematics and Statistics, \url{grimmer@jhu.edu}}}

    \date{}
    \maketitle
    \begin{abstract}
        We extend recent computer-assisted design and analysis techniques for first-order optimization over structured functions---known as performance estimation---to apply to structured sets. We prove ``interpolation theorems'' for smooth and strongly convex sets with interior point conditions and bounded diameter, showing a wide range of extremal questions amount to structured mathematical programs. Prior function interpolation theorems are recovered as a limit of our set interpolation theory. Our theory provides finite-dimensional formulations of performance estimation problems for algorithms utilizing separating hyperplane oracles and linear optimization oracles of smooth/strongly convex sets. As applications of this computer-assisted machinery, we identify a minimax optimal separating hyperplane method and areas for improvement in the theory of Frank-Wolfe and non-Lipschitz Smooth Optimization.
    \end{abstract}

    \section{Introduction} \label{Sect:Intro}
Given a proposed algorithm, the study of its worst-case convergence guarantees over some family of problem instances can be framed as a meta-optimization problem. In this work, we will consider examples of first-order optimization methods applied to problems from a structured family of objective functions and constraint sets. Given some measure of performance (e.g., final objective gap), the exact worst-case performance of the method corresponds to computing the maximum of this measure at the algorithm's output over all problem instances. This is known as the Performance Estimation Problem (PEP). In general, optimizing over a family of functions and sets is an infinite-dimensional problem, typically beyond direct approach. Surprisingly, recent advances have shown that PEPs can often be computer-solved and have led to several state-of-the-art results in the design and analysis of optimization algorithms.

Performance estimation was first proposed by Drori and Teboulle in~\cite{FirstPEP}, where tractable relaxations of these worst-case PEPs were considered. Shortly afterward, the ``interpolation theorems'' of Taylor, Hendrickx, and Glineur\edit{~\cite{Interpolation,Interpolation2}} proved that for many gradient methods and structured function classes, the PEP exactly corresponds to a tractable finite-dimensional problem. 
Specifically, their function interpolation results showed that the infinite-dimensional optimization problem of finding a worst-case problem instance for $N$ steps of a given gradient method could be equivalently reformulated as a nonconvex Quadratically Constrained Quadratic Program (QCQP), which can often then be reformulated as a Semidefinite Program (SDP). As a result, for a modest fixed number of algorithm steps $N$, computers can be used to determine an algorithm's worst-case problem instances and, dually, to produce optimal convergence proofs. \edit{Similar techniques from the perspective of control were developed by~\cite{Lessard2016}.} These computer-assisted techniques provide an invaluable starting point for formally proving improved theory for general $N$.

\edit{Since this initial application, performance estimation theory has been applied to many different settings and algorithms, including conditional gradient~\cite{Interpolation2}, proximal gradient~\cite{Interpolation2,ProxGradientPEP_Taylor,AcceleratedProx_Kim}, alternating projections~\cite{Interpolation2,zamani2024}, stochastic gradient descent~\cite{Taylor_InterpolationStochastic}, line-search~\cite{ExactLinesearchPEP_Klerk,Drori2019}, strong convexity~\cite{Interpolation,ITEM_Taylor,Drori2019}, fixed point iteration~\cite{FixedPointPEP_Lieder,OptimalFixedPoint_Park}, and much more. See the growing collection of examples documented in the PEPit toolbox~\cite{pepit2022}\footnote{Available at \url{https://pepit.readthedocs.io/en/latest/examples.html}.}.
}

\edit{Computer PEP solutions directly produce tight numerical convergence guarantees and, in many cases, have led to tight analytical convergence rates. They also} have a record of directly improving the state of the art in algorithm design by enabling the optimization of parameter selections. In general, the task of finding algorithmic parameters (e.g., stepsizes and momentum sequences) minimizing the worst-case PEP guarantee is a nonconvex problem\footnote{A branch-and-bound solver for this hard task was developed by Das Gupta et al.~\cite{ShuvoBnB}.}. \edit{Despite this apparent hardness, many recent works have found success at this task, often by observing structure in numerical results and devising corresponding analytical solutions. Some such advances include:
\begin{itemize}[noitemsep,topsep=0pt]
    \item For smooth convex optimization, Kim and Fessler~\cite{OGM} derived an Optimized Gradient Method, improving on Nesterov's accelerated method by constants and attaining the minimax optimal rate by Drori~\cite{Drori_OGMOptimal}. In~\cite{OGMG}, they proposed an optimized method (OGM-G) for minimizing gradient norm.
    \item For smooth strongly convex optimization, Taylor and Drori~\cite{ITEM_Taylor} derived an exactly optimal method (ITEM) improving upon the asymptotically optimized triple momentum method of~\cite{TripleMomentum_Scoy}.
    \item Altschuler and Parrilo~\cite{Silver_Accel_General,Silver_Accel} and Grimmer, Shu, and Wang~\cite{Grimmer_Accel_Original,Grimmer_Accel} showed that gradient descent with fractal stepsizes provably accelerates in big-O.
\end{itemize}
These advances} on long-standing, best-known results are in large part due to the insights generated from performance estimation and computer assistance.

The primary aim of this work is the development of new interpolation theorems for smooth and strongly convex sets, potentially with \edit{interior point conditions} and bounded diameter, paralleling the function-oriented theorems of Taylor et al.~\cite{Interpolation,Interpolation2}. Prior function interpolation theorems in~\cite{Interpolation2} can be used to model simply convex constraint sets \edit{but do not capture additional structural properties of set smoothness and strong convexity. Our theorems enable the formulation of computer-solvable PEPs over these structured families of constraint sets.}
Formal definitions of smoothness and strong convexity of closed convex sets and functions (and several equivalent conditions) are given in \cref{Sec:Preliminaries}. In brief, a closed convex function $f$ is smooth and strongly convex if each subgradient provides certain quadratic upper and lower bounds on $f$, respectively. Similarly, a closed convex set $C$ is smooth and strongly convex if each normal vector provides certain balls inner and outer approximating $C$.

\paragraph{Our Contributions.}
Our primary contribution is a collection of interpolation theorems for structured sets, \edit{extending the prior theory of \cite{Interpolation2} for convex sets}.
Our theorems provide tractable, finite representations for a variety of settings, including smooth and strongly convex sets, sets with bounded diameters, and sets with \edit{interior point conditions}. See Theorems~\ref{Thm:Interpolation_Necc}-\ref{Thm:Interpolation_NoSlack}. These interpolations allow us to express the worst-case performance of many constrained optimization algorithms as numerically solvable performance estimation problems.
\edit{\cref{Thm:SetFuncInterp_Smooth} relates our theory to prior function interpolation theory.} 

\vspace{0.2cm}

\noindent {\it Outline.}
\cref{Sec:Preliminaries} defines basic notations and presents the function interpolation theorems of~\cite{Interpolation}. \cref{Sec:Interpolation} then presents our main result: interpolation theorems for structured sets. The remainder of this paper presents three applications, \edit{showing computational\footnote{All numerical results were obtained with \texttt{Mosek}~\cite{mosek} or \texttt{Gurobi}~\cite{gurobi} via \texttt{JuMP}~\cite{JuMP}. Code for reproducibility is available at \href{https://github.com/alanluner/PEPStructuredSets}{\texttt{github.com/alanluner/PEPStructuredSets}}.} and modeling potential of our theory:
\begin{itemize}
    \item \cref{Sec:SepHyperplane} shows, given a separating hyperplane oracle for a set $C$, our interpolation theorems allow one to determine any proposed method's convergence to feasibility as a semidefinite program. Matching resulting numerics, we derive a simple, minimax optimal, separating hyperplane algorithm.
    \item \cref{Sec:FW} shows, for constrained optimization, that our interpolation theorems enable the performance of the Frank-Wolfe method to be modeled as a semidefinite program with separable nonconvex equality constraints. By solving such nonconvex problems (globally and locally), we survey the potential for the Frank-Wolfe method to accelerate given structured sets.
    \item \cref{Sec:EpiSmooth} shows, for unconstrained optimization, that our interpolation theorems allow us to analyze optimization of functions with a smooth epigraph. Preliminary results, solving semidefinite programs with a rank constraint, show potential to generalize current accelerated theory. 
\end{itemize}
}

    \section{Definitions and Preliminaries}\label{Sec:Preliminaries}
\cref{Subsec:func-def-interpol} first introduces standard definitions of smooth and strongly convex functions and the function interpolation of Taylor et al.~\cite{Interpolation}. \cref{Subsec:set-def-interpol} then introduces the analogous concepts for sets in preparation for our main results.

\subsection{Preliminaries on Smooth and Strongly Convex Functions}\label{Subsec:func-def-interpol}
Consider a closed convex function $f\colon \mathbb{R}^d\rightarrow \mathbb{R}\cup \{\infty\}$. We denote $\langle\cdot,\cdot\rangle$ and $\|\cdot\|$ for the Euclidean inner product and associated norm. We denote the domain of $f$ by $\mathrm{dom\ }f$ and its subdifferential at some $x\in\mathrm{dom\ }f$ by $\partial f(x) = \{g\mid f(y) \geq f(x)+\langle g, y-x\rangle \quad  \forall y\in\mathbb{R}^d\}$. Each element $g\in\partial f(x)$ is referred to as a subgradient and must be the gradient $\nabla f(x)$ when $f$ is differentiable.

Typically, $L$-smoothness of a function is defined as $\nabla f(x)$ being $L$-Lipschitz. For the sake of developing other symmetries, here we instead (equivalently) define a closed convex $f$ as being {\bf $L$-smooth} if for all $x,y \in \R^d$ and $g\in\partial f(x)$,
\begin{equation}    \label{Eqn:def-sm-func}
    f(y) \leq f(x) + \langle g, y-x \rangle + \frac{L}{2}\|y-x\|^2 \ .
\end{equation}
That is, each subgradient provides a quadratic upper bound on $f$. Mirroring this, we say $f$ is {\bf $\mu$-strongly convex} if for all $x,y \in \R^d$ and $g\in\partial f(x)$,
\begin{equation} \label{Eqn:def-sc-func}
    f(y) \geq f(x) + \langle g, y-x \rangle + \frac{\mu}{2}\|y-x\|^2 \ . 
\end{equation}
That is, each subgradient provides a quadratic lower bound on $f$.
We allow the extreme cases of $L,\mu \in \{0,\infty\}$, taking limits above. Every convex $f$ is vacuously $\infty$-smooth and $0$-strongly convex. The only $0$-smooth functions are linear and the only $\infty$-strongly convex functions are indicators of points.
Denote the set of all $\mu$-strongly convex and $L$-smooth functions $f$ by $\mathcal{F}_{\mu,L}$.

As a motivating example of the performance estimation problem, consider minimizing some \edit{$f \in \mathcal{F}_{0,L}$} via $N$ \edit{steps of a gradient method, defined as
\begin{equation} \label{Eqn:GM}
    x_{i+1} = x_i - \frac{1}{L} \sum_{j=0}^i H_{i,j} \nabla f(x_j) \tag{GM}
\end{equation}
with initial point $x_0$ and a lower triangular matrix $H$. Note this model captures simple methods like gradient descent via a diagonal $H$ and more complex momentum methods like Nesterov's accelerated method~\cite{Nesterov} and the optimal gradient method of~\cite{OGM}.}
Suppose as an initial condition, we are guaranteed the initial distance to a minimizer $\|x_0-x_\star\|$ is \edit{bounded above by $R$}. Then the infinite-dimensional PEP of seeking a problem instance with objective in \edit{$\mathcal{F}_{0,L}$} achieving the largest final objective gap can be formulated as
\begin{equation}\label{Eqn:basicGD-example1}
\edit{
    p_\mathrm{S}(L,R) := \begin{cases}
    \max\limits_{x_i,f} & f(x_N) - f(x_\star)\\
    \mathrm{s.t.} & f\in \mathcal{F}_{0,L}\\
    & x_{i+1} = x_i - \frac{1}{L} \sum_{j=0}^i H_{i,j} \nabla f(x_j) \\
    & \|x_0-x_\star\| \leq R, \quad \nabla f(x_\star) = 0 \ .
\end{cases}
}
\end{equation}
Note that aside from the constraint \edit{$f\in \mathcal{F}_{0,L}$}, this optimization problem only depends on $f$ via its function value and gradient at finitely many points, namely $\{x_0,\dots, x_N, x_\star\}$.
Denote these values by $f_i=f(x_i)$ and $g_i=\nabla f(x_i)$ for each $i\in\{0,\dots, N,\star\}$.
The following definition of function interpolation from~\cite{Interpolation} allows us to reformulate into these more limited quantities:
\begin{definition}\label{Def:FuncInterpolability}
     Consider a set of observations $S = \{(x_i, g_i, f_i)\}_{i \in \I}$ where $x_i,g_i \in \R^d$ and $f_i \in \R$ for all $i \in \I$. The set $S$ is $\mathcal{F}_{\mu,L}$\textbf{-interpolable} if there exists a function $f \in \F_{\mu,L}$ such that $g_i \in \partial f(x_i)$ and $f(x_i) = f_i$ for all $i \in \I$.
\end{definition}

Function interpolability describes whether there exists an appropriate function that matches the given finite set of function values and subgradients and interpolates between them (See \cref{Fig:FunctionSetInterpolation}). In terms of the working example~\eqref{Eqn:basicGD-example1}, this can then be expressed finitely as
\begin{equation}\label{Eqn:basicGD-example2}
    \edit{
    p_\mathrm{S}(L,R) = \begin{cases}
    \max\limits_{x_i,g_i,f_i} & f_N - f_\star\\
    \mathrm{s.t.} & \{(x_i, g_i, f_i)\}_{i \in \{0,\dots,N,\star\}} \text{ is }\mathcal{F}_{0,L}\text{-interpolable}\\
    & x_{i+1} = x_i - \frac{1}{L} \sum_{j=0}^i H_{i,j}  g_j \\
    & \|x_0-x_\star\| \leq R, \quad g_\star=0 \ .
\end{cases}
}
\end{equation}

Although finite, this formulation is still not particularly tractable as the interpolability constraint is, at first glance, not approachable. The critical insight of Taylor et al.~\cite[Theorem 4]{Interpolation} was establishing explicit necessary and sufficient conditions for function interpolability.
\begin{theorem}\label{Thm:func-interp} \emph{(\!\!\cite[Theorem 4]{Interpolation})}
    A set $S = \{(x_i, g_i, f_i)\}_{i \in \I}$ is $\F_{\mu,L}$-interpolable if and only if the following quadratic condition holds for all $i,j \in \I$:
    \begin{equation}    \label{Eqn:func-interp-constraint}
        f_i - f_j + \langle g_i, x_j - x_i \rangle + \frac{1}{2(1-\mu/L)}\left(\frac{1}{L} \|g_i - g_j \|^2 + \mu \|x_i - x_j \|^2 - \frac{2\mu}{L}\langle g_i - g_j, x_i - x_j \rangle \right) \leq 0 \ .
    \end{equation}
\end{theorem}
Plugging this into~\eqref{Eqn:basicGD-example2} yields an explicit QCQP exactly describing gradient descent's worst-case behavior over the considered family of structured functions. Semidefinite programming lifting techniques (and duality) can be further used to gain insights from this new formulation.

\edit{
\begin{remark} \label{Rem:DimensionBound}
    Throughout the literature and in this work, the \textit{equivalent} reformulation of the QCQP into a semidefinite program depends on an assumption of high dimension (e.g., $d \geq N+2$). This assumption is only required for the SDP to be an equivalent problem, rather than a relaxation. For any $d$, the resulting SDP provides a valid, dimension-independent upper bound.
\end{remark}
}

It is worth noting that there are many collections of necessary constraints on the set of observation data $S$. Any standard inequality for $\mu$-strongly convex, $L$-smooth functions must necessarily hold: for example, the definitions~\eqref{Eqn:def-sm-func} and~\eqref{Eqn:def-sc-func}. However, as Taylor et al.~\cite{Interpolation} highlight, these other inequalities prove not to be sufficient for interpolation. Hence, identification of the right inequalities, e.g.,~\eqref{Eqn:func-interp-constraint}, is key for the ``if and only if'' nature of interpolation theory and, consequently, its power in enabling computer assistance. Similarly, in \cref{Sec:Interpolation}, we find that while there are many necessary constraints for set interpolability, a careful selection is needed for sufficiency.

\subsection{Preliminaries on Smooth and Strongly Convex Sets} \label{Subsec:set-def-interpol}

We now switch our focus to smoothness and strong convexity of sets. These notions mirror their function counterparts; see~\cite{Vial_StrongConvexity,Vial_StrongAndWeak,Goncharov2017} as classic references. These properties are, however, quite under-explored in the contemporary first-order method literature compared to their function counterparts, which are the backbone of much modern convergence theory. In part, this work aims to provide tools to eventually alleviate this discrepancy.

Consider a nonempty closed convex set $C\subseteq \mathbb{R}^d$.
We denote the interior of a set $C$ by $\interior C$ and its boundary by $\bdry C$. We denote the normal cone of $C$ at some $z \in \bdry C$ by $N_C(z) = \{v \mid \langle v, x-z\rangle \leq 0 \quad  \forall x \in C \}$ and refer to individual elements as normal vectors. Let $B(x,r)$ be the closed ball of radius $r$ centered at $x$. We define $\diam(C) = \sup\{\|x-y\| \mid x,y \in C \}$ and the $\delta$-interior of $C$ by $\Interior{\delta} C = \{x \mid B(x,\delta) \subseteq C\}$. Note $\Interior{0} C = C$.
Lastly, we define the Minkowski sum of two sets $C_1$ and $C_2$ as
\begin{equation*}
    C_1 + C_2 = \{x+y \mid x\in C_1, y \in C_2 \} \ .
\end{equation*}

We now define smoothness and strong convexity with respect to sets. For a more full discussion of the equivalent definitions of these properties and their analogs with smooth and strongly convex functions, we refer readers to \cite[Section 2]{LiuGrimmer}. Overall, one can recover the analogous definitions for functions by replacing unit normal vectors with gradients, bounding balls with bounding quadratics, and sets with epigraphs.
A set $C$ is $\beta$\textbf{-smooth} if for any $z \in \bdry C$ and unit vector $n \in N_C(z)$,
\begin{equation*} \label{Eqn:SmoothSet}
    B(z - \ob n, \ob) \subseteq C \ . 
\end{equation*}
A set $C$ is $\alpha$\textbf{-strongly convex} if for any $z \in \bdry C$ and unit vector $n \in N_C(z)$,
\begin{equation*} \label{Eqn:SCSet}
    B(z - \oa n, \oa) \supseteq C \ .
\end{equation*}
\noindent Again, we allow the limiting cases of $\alpha,\beta\in\{0,\infty\}$. Every convex set is $0$-strongly convex and $\infty$-smooth. The only $\infty$-strongly convex sets are singletons and the only $0$-smooth sets are halfspaces.
Let $\mathcal{C}_{\alpha,\beta,D}$ denote the set of all closed convex sets $C$ that are $\alpha$-strongly convex, $\beta$-smooth, and have $\diam(C) \leq D$. See an illustration of smooth and strongly convex sets in \cref{Fig:sm-sc-set-examples}.

\edit{Our discussion of set smoothness and strong convexity throughout will focus on these properties at full dimension. One could consider smoothness/strong convexity with respect to a fixed subspace. Through a simple change of variables, our definitions and theory could be applied.}

\begin{figure}
    \centering
    \scalebox{0.6}{\input{Tikz_SmoothAndSCSets_WithC0}}
    \caption{Example closed convex sets with red inner approximating balls from smoothness and blue outer approximating balls from strong convexity. Namely, (a) is $\beta$-smooth but not strongly convex, (b) is $\beta$-smooth and $\alpha$-strongly convex, and (c) is $\alpha$-strongly convex but not smooth. \edit{The red dotted lines indicate the set $C_0$ from \cref{Prop:SmoothProperties} and the blue dotted lines indicate the set $C_0$ from \cref{Prop:SCProperties}. The blue dot marks the origin where applicable.}}
    \label{Fig:sm-sc-set-examples}
\end{figure}

Below are equivalent characterizations of set smoothness and strong convexity that are useful in our analysis, paralleling classic results for functions.
\begin{proposition}   \label{Prop:SmoothProperties}
    \emph{(\!\!\cite[Proposition 3]{LiuGrimmer})}
    For any \edit{nonempty} closed convex set $C$, the following are equivalent:
    \begin{enumerate}
        \item $C$ is $\beta$-smooth.
        \item For any $z_1, z_2 \in \bdry C$, with unit normals $n_1 \in N_C(z_1)$ and $n_2 \in N_C(z_2)$, 
        \begin{equation*}
            \langle z_1-z_2, n_1 - n_2 \rangle \geq \ob \| n_1 - n_2 \|^2 \ .
        \end{equation*}
        \item There exists a closed convex set $C_0$ such that $C_0 + B(0,\ob) = C$.
    \end{enumerate}
\end{proposition}
\begin{proposition} \label{Prop:SCProperties}
\emph{(\!\!\cite[Theorem 2.1]{Goncharov2017})}
    For any \edit{nonempty} closed convex set $C$, the following are equivalent:
    \begin{enumerate}
        \item $C$ is $\alpha$-strongly convex.
        \item For any $z_1, z_2 \in \bdry C$, with unit normals $n_1 \in N_C(z_1)$ and $n_2 \in N_C(z_2)$,
        \begin{equation*}
                \langle z_1-z_2, n_1 - n_2 \rangle \leq \oa \| n_1 - n_2 \|^2 \ .
        \end{equation*}
        \item There exists a closed convex set $C_0$ such that $C_0 + C = B(0,\oa)$.
    \end{enumerate}
\end{proposition}

\noindent Many more equivalent characterizations of strongly convex sets are given in~\cite{Goncharov2017}. The following pair of lemmas will also help develop our theory.
\begin{lemma}\label{Lem:SmoothDiam}
    A set $C$ is $\beta$-smooth with $\diam(C) \leq D$ if and only if $C = C_0 + B(0,\ob)$ for some convex set $C_0$ with $\diam(C_0) \leq D-\frac{2}{\beta}$.
\end{lemma}
\begin{proof}
    We know from \cref{Prop:SmoothProperties} that $C$ is $\beta$-smooth if and only if there exists a convex set $C_0$ such that $C = C_0 + B(0,\ob)$. So the result follows as\belowdisplayskip=-12pt
    \begin{align*}
        \diam(C) & = \sup \{ \|x + \ob n_x - (y + \ob n_y) \| \mid x,y \in C_0, \|n_x\| \leq 1, \|n_y\| \leq 1 \} \\
        & = \sup \{ \|x - y\| \mid x,y \in C_0 \} + \frac{2}{\beta}  = \diam(C_0) + \frac{2}{\beta} \ .
    \end{align*} \myqed
\end{proof}

\begin{lemma}\label{Lem:SmoothSCParam}
    A set $C$ is $\alpha$-strongly convex and $\beta$-smooth if and only if $C = C_0 + B(0,\ob)$ for some $\gamma$-strongly convex set $C_0$, with $\gamma := (\oa - \ob)^{-1} = \frac{\alpha \beta}{\beta - \alpha}$.
\end{lemma}
\begin{proof}
    This follows from property 3 in both Propositions~\ref{Prop:SmoothProperties} and~\ref{Prop:SCProperties}.
    Namely, $C$ is $\alpha$-strongly convex and $\beta$-smooth if and only if there exist convex sets $C_0^\mathtt{sc}$ and $C_0^\mathtt{sm}$ such that $C + C_0^\mathtt{sc} = B(0,\oa)$ and $C = C_0^\mathtt{sm} + B(0,\ob)$. Together, these properties imply
    $C_0^\mathtt{sm} + B(0,\ob) + C_0^\mathtt{sc} = B(0,\oa)$,
    and so $C_0^\mathtt{sm} + C_0^\mathtt{sc} = B(0,\frac{1}{\gamma})$, i.e., $C_0^\mathtt{sm}$ is $\gamma$-strongly convex.
    Conversely, if $C = C_0 + B(0,\ob)$ for some $\gamma$-strongly convex $C_0$, it immediately follows that $C$ is $\beta$-smooth. Further some convex $C_0'$ must exist with $C_0+C_0' = B(0,\frac{1}{\gamma})$. Hence $C + C_0' = B(0,\oa)$, i.e., $C$ is $\alpha$-strongly convex. \myqed
\end{proof}

Finally, we define set interpolability, paralleling the function interpolability of \cref{Def:FuncInterpolability}. In addition to smoothness and strong convexity, we model two additional structures: diameter bounds and various \edit{interior point conditions}, which occur widely in constrained optimization theory.

We consider a set of observations $S$ defined by two types of data: points $z_i\in \R^d$ on the boundary of the constraint set with nonzero normal vectors $v_i\in \R^d$ and points $x_k\in \R^d$ in the $\delta_k$-interior of the set. Note the index sets $i\in \I$ and $k\in\K$ need not be of the same size or even both be nonempty. Given target strong convexity $\alpha$, smoothness $\beta$, and diameter $D$, we can formalize whether given observation data can be interpolated by some such set.
\begin{definition}\label{Def:SetInterpolability}
    Consider a set of observations $S=\SSetInterp$, with $z_i,v_i, x_k \in \R^d$ and $\delta_k \in \R_{\geq 0}$ for all $i \in \I$ and all $k \in \K$. The set $S$ is {\bf $\C_{\alpha,\beta,D}$-interpolable} if there exists $C \in \mathcal{C}_{\alpha,\beta,D}$ such that $z_i \in C$ and $v_i \in N_C(z_i)$ for all $i \in \I$ and $x_k \in \Interior{\delta_k} C$ for all $k \in \K$.
\end{definition}
\noindent Note that this definition allows $\delta_k = 0$ (denoting the corresponding $x_k$ simply must be a member of $C$) and all limiting values $\alpha,\beta,D \in\{0,\infty\}$.

This work's primary contribution is identifying computationally tractable, verifiable equivalent conditions for set interpolability. \edit{Our theory builds upon the interpolation theory over convex sets introduced in~~\cite[Theorem 3.6]{Interpolation2}. 

Before presenting our theory in \cref{Sec:Interpolation}, we consider a simplified setting where $\alpha = 0$ and $\beta = \infty$. In this case, interpolability corresponds to determining if a bounded diameter convex set exists with the required normals $v_i\in N_C(z_i)$ and $\delta_k$-interior points $x_k$.
When the interpolation conditions introduced in \cref{Sec:Interpolation} hold, a simple, explicit construction of the interpolating convex, bounded set exists: $C = \mathrm{conv}( \{z_i \}_{i \in \I}, \{B(x_k,\delta_k)\}_{k \in \K} )$.}
\cref{Fig:FunctionSetInterpolation} shows an example of this simplified case's construction. Designing and analyzing such a construction is the primary step in proving the sufficiency of given interpolation conditions. Our constructions grow substantially in complexity from this easy case, but remain explicit. As a result, one can always recover a (smooth, strongly convex, bounded diameter) set that exactly interpolates any observed normal vectors and interior points satisfying our theory's conditions.

\begin{figure}
    \centering
    \scalebox{0.45}{\input{Tikz_FunctionSetInterpolation}}
    \vspace{-7mm}
    \caption{Left: Function interpolation of function values $f_i$ and subgradients $g_i$ evaluated at $x_i$. Right: Set interpolation of boundary points $z_i$ with normals $v_i$ and $\delta_k$-interior points $x_k$.}
    \label{Fig:FunctionSetInterpolation}
\end{figure}

    \section{Main Result: Set Interpolation Theory} \label{Sec:Interpolation}
This section presents and proves our main results on set interpolability.

We first define a set of conditions that will be necessary throughout our theory. Consider a set of observations $S=\SSetInterp$ and denote the associated unit normal vectors by $n_i = \unitv{i}{}$ for all $i \in \I$.
Due to the added complexity of our set interpolation model, our conditions involve the introduction of auxiliary parameters $\{w_k\}_{k \in \K} \subset \R^d$ along with a constant factor \edit{$\lambda\in [\frac{1}{\sqrt{2}},1]$}, typically equal to one, enabling strengthened versions of these conditions.
We find that for any $\mathcal{C}_{\alpha,\beta,D}$-interpolable observation data $S=\SSetInterp$, there must exist auxiliary parameters $\{w_k\}_{k \in \K}$ \edit{and a choice of $\lambda$} such that for all $i,j \in \I$ and  $k,\el \in \K$,
\begin{align}
    & \|z_i-\oa n_i - (z_j - \ob n_j) \| \leq \og \tag{Interp1} \label{Eqn:Condition1} \\
    & \| z_i - \oa n_i - w_k \| \leq \og - s_k \tag{Interp2} \label{Eqn:Condition2} \\
    & \|x_k - w_k \| \leq \ob - \delta_k + s_k \tag{Interp3} \label{Eqn:Condition3} \\
    & \| z_i - \ob n_i - (z_j - \ob n_j) \| \leq \lambda(D-\frac{2}{\beta}) \tag{Interp4} \label{Eqn:Condition4} \\
    & \|z_i - \ob n_i - w_k \| \leq \lambda(D-\frac{2}{\beta}) - s_k \tag{Interp5} \label{Eqn:Condition5}  \\
    & \|w_k - w_\el \| \leq \lambda(D-\frac{2}{\beta}) - s_k - s_\el \tag{Interp6} \label{Eqn:Condition6}
\end{align}
where $\gamma = (\oa - \ob)^{-1}$ and $s_k = \max\{0,\delta_k - \ob \}$ for all $k \in \K$. In the case $\alpha = \beta$, we will use the convention that $\og = 0$.
As a shorthand, we write $S \in \mathrm{Interp}(\alpha, \beta, D; \lambda)$ if $\{w_k\}_{k \in \K} \subset \R^d$ exist satisfying all six conditions.

\edit{We can briefly motivate our auxiliary parameters $w_k$ and $\lambda$. The values $w_k$ arise from interpolating a smooth set, which can be decomposed into $C = C_0 + B(0,\og)$. Effectively, for each interior point $x_k \in C$, the point $w_k$ represents a corresponding point in $C_0$. The auxiliary parameter $\lambda$ arises from possible slack in bounding the diameter of a strongly convex set (See \cref{Rem:SCSlack}). Note that if a condition holds for some $\lambda$, then it holds for all $\hat{\lambda} > \lambda$. Lastly, we emphasize the normalized vectors $n_i$. In \cref{Sec:FW}, we will see that having simultaneous constraints on $v_i$ and $\frac{v_i}{\|v_i\|}$ results in more computationally difficult PEP formulations.}

\edit{If we simplify our set requirements, the conditions above simplify in turn. If one only enforces smoothness and strong convexity (i.e., $\K = \emptyset, D = \infty$), only a single condition \eqref{Eqn:Condition1} is needed, mirroring the single interpolation condition \eqref{Eqn:func-interp-constraint} of Taylor et al.~\cite{Interpolation}. The other conditions are only applicable depending on the constraints on the set structure. We summarize this dependency in~\cref{Tbl:SummaryTable}. In particular, the auxiliary variables $w_k$ are only necessary for smooth sets ($\beta < \infty$) with interior point conditions ($\K \neq \emptyset$); otherwise, the condition reduces to $x_k = w_k$.}

\begin{table}
    \centering\footnotesize
    \begin{tabular}{|c|c|c|c|c|c|}
    \hline
    \eqref{Eqn:Condition1} & \eqref{Eqn:Condition2} & \eqref{Eqn:Condition3} & \eqref{Eqn:Condition4} & \eqref{Eqn:Condition5} & \eqref{Eqn:Condition6} \\
    \hline
     & & $\beta< \infty$ & & $D < \infty$ & $D < \infty$ \\
    All settings & $\K \neq \emptyset$ & \textbf{and} & $D < \infty$ & \textbf{and} & \textbf{and} \\
     & & $\K \neq \emptyset$ & & $\K \neq \emptyset$ & $\K \neq \emptyset$ \\
    \hline
    \end{tabular}
    \caption{\edit{Summary of settings in which each interpolation condition is applicable.}
    }
    \label{Tbl:SummaryTable}
\end{table}

Our three main theorems below show that these conditions are always necessary for interpolation (\cref{Thm:Interpolation_Necc}); they become sufficient if the diameter bound is tightened by a factor of at most $\sqrt{2}$ (\cref{Thm:Interpolation_SuffWithSlack});
they are necessary and sufficient if either the strong convexity or diameter bound is omitted (\cref{Thm:Interpolation_NoSlack}). 
Immediately afterward, we discuss the necessity of this $\sqrt{2}$ gap between our first two theorems and provide several convenient corollaries. \cref{Fig:GeneralConstruction} presents sample constructions for convex, smooth, and strongly convex interpolating sets for the case where $\K$ is empty. These examples are special cases of the constructions used in our sufficiency proofs.

The following theorems, proven in the following subsections, consider a set of observations $S = \SSetInterp$. Recall $d$ denotes the dimension of space being considered, i.e., $z_i,v_i,x_k \in \R^d$. 
\begin{theorem} \label{Thm:Interpolation_Necc}
    If $S$ is $\mathcal{C}_{\alpha,\beta,D}$-interpolable, then $S \in \mathrm{Interp}(\alpha, \beta, D; 1)$.
\end{theorem}
\begin{theorem} \label{Thm:Interpolation_SuffWithSlack}
    If $S \in \mathrm{Interp}(\alpha,\beta,D; \sqrt{\frac{d+1}{2d}})$, then $S$ is $\mathcal{C}_{\alpha,\beta,D}$-interpolable.
    In particular, if $S \in \mathrm{Interp}(\alpha,\beta,D; \frac{1}{\sqrt{2}})$, then $S$ is $\mathcal{C}_{\alpha,\beta,D}$-interpolable.
\end{theorem}
\begin{theorem} \label{Thm:Interpolation_NoSlack}
    If $\alpha=0$ or $D \geq \frac{2}{\alpha}$, then $S$ is $\mathcal{C}_{\alpha,\beta,D}$-interpolable if and only if $S \in \mathrm{Interp}(\alpha, \beta, D; 1)$.
\end{theorem}

\begin{figure}
    \centering
    \scalebox{0.42}{\input{Tikz_GeneralSetConstruction}}
    \caption{Examples of constructed interpolations with no \edit{required interior points} ($\K=\emptyset$).}
    \label{Fig:GeneralConstruction}
\end{figure}

\edit{
\begin{remark}[On Indicator Function Formulation]
    Instead of using boundary points and normal vectors, one can equivalently define set interpolation in terms of indicator functions. For any set $C$ and indicator $\iota_C = \begin{cases} 0 &\text{if } x \in C \\
        \infty & \text{otherwise},
        \end{cases}$
    one has that $v_i \in N_C(z_i)$ if and only if $v_i \in \partial \iota_C(z_i)$. Similarly, one can define a $\delta$-interior point as some $x$ such that $\iota_C(x+\delta \zeta) = 0$ for all $\|\zeta\| \leq 1$. We focus on the normal-vector formulation to align with our later applications.
\end{remark}
}

\begin{remark}[On the Computational Cost of Auxiliary Variables]
    The function interpolation result of \cref{Thm:func-interp} provides a direct means to algebraically check if a collection of observation data $\{(x_i,f_i,g_i)\}_{i \in \I}$ is interpolable. One must check $O(|\I|^2)$ quadratic inequalities. Our set interpolation theory does not present such an easy algebraic check. One must determine whether $\{w_k\}_{k \in \K}$ exist satisfying the needed conditions. This corresponds to solving a second-order cone program with $O(|\K|)$ variables and $O(|\I|^2+|\K|^2)$ constraints. \edit{However, in cases with no or very few interior point conditions ($|\K|=0$ or small), or where $\beta = \infty$, this cost reduces back to checking $O(|\I|^2)$ inequalities.}
\end{remark}
\begin{remark}[On the Tightness of the Gap Between Theorems~\ref{Thm:Interpolation_Necc} and~\ref{Thm:Interpolation_SuffWithSlack}] \label{Rem:SCSlack}
    The gap between $\lambda=1$ and $\lambda=\sqrt{\frac{d+1}{2d}}\geq 1/\sqrt{2}$ in Theorems~\ref{Thm:Interpolation_Necc} and~\ref{Thm:Interpolation_SuffWithSlack} is fundamental: Consider the regular $d$-simplex in $\R^{d+1}$, with $d>1$. Placing the centroid at 0, our vertices become $(\frac{d}{d+1}, \frac{-1}{d+1}, \dots, \frac{-1}{d+1})$, $(\frac{-1}{d+1}, \frac{d}{d+1}, \frac{-1}{d+1} \dots, \frac{-1}{d+1})$, etc. Taking $\I = [0\mathrm{\,:\,}d]$, we choose $\{z_i\}_{i \in \I}$ to be these $d+1$ vertices and $v_i = \unit{z_i}$ for all $i \in \I$. Further, set $\K = \emptyset$, \edit{$S = (\{(z_i,v_i)\}_{i \in \I}, \emptyset )$}, \edit{and $\alpha = \sqrt{\frac{d+1}{d}}$, $\beta = \infty$, $D = \sqrt{2}$.} Observing that
    \begin{equation*}
        \edit{
        \|z_i - \oa n_i - (z_j - \ob n_j)\| = \| z_j \| = \sqrt{\frac{d}{d+1}} \ , \qquad \qquad \|z_i - z_j \| = \sqrt{2}
        }
    \end{equation*}
    for all $i\neq j$, \eqref{Eqn:Condition1} and \eqref{Eqn:Condition4} are satisfied \edit{with $\lambda=1$}, so $S \in \mathrm{Interp}(\alpha,\beta,\sqrt{2}; 1)$.
    \edit{However, since all $z_i$ satisfy $\|z_i\| = \oa$, $B(0,\oa)$ is the unique $\alpha$-strongly convex (and vacuously $\beta$-smooth) set that interpolates $S$.} So  $S$ is not $\C_{\alpha,\beta,\sqrt{2}}$-interpolable as any interpolating set must have diameter $2\sqrt{\frac{d}{d+1}} > \sqrt{2}$, exactly matching the gap in our theorems. \edit{See \cref{Fig:EquilateralTriangle} demonstrating a similar example in $\R^2$}.
\end{remark}

\begin{figure}
    \centering
    \scalebox{0.7}{\tikzset{every picture/.style={line width=0.75pt}} 

\begin{tikzpicture}[x=0.75pt,y=0.75pt,yscale=-1,xscale=1]

\draw  [fill={rgb, 255:red, 0; green, 0; blue, 0 }  ,fill opacity=1 ] (147.5,80) .. controls (147.5,78.62) and (148.62,77.5) .. (150,77.5) .. controls (151.38,77.5) and (152.5,78.62) .. (152.5,80) .. controls (152.5,81.38) and (151.38,82.5) .. (150,82.5) .. controls (148.62,82.5) and (147.5,81.38) .. (147.5,80) -- cycle ;
\draw  [fill={rgb, 255:red, 0; green, 0; blue, 0 }  ,fill opacity=1 ] (222.5,209.9) .. controls (222.5,208.52) and (223.62,207.4) .. (225,207.4) .. controls (226.38,207.4) and (227.5,208.52) .. (227.5,209.9) .. controls (227.5,211.28) and (226.38,212.4) .. (225,212.4) .. controls (223.62,212.4) and (222.5,211.28) .. (222.5,209.9) -- cycle ;
\draw  [fill={rgb, 255:red, 0; green, 0; blue, 0 }  ,fill opacity=1 ] (72.5,209.9) .. controls (72.5,208.52) and (73.62,207.4) .. (75,207.4) .. controls (76.38,207.4) and (77.5,208.52) .. (77.5,209.9) .. controls (77.5,211.28) and (76.38,212.4) .. (75,212.4) .. controls (73.62,212.4) and (72.5,211.28) .. (72.5,209.9) -- cycle ;
\draw    (225,209.9) -- (260.77,230.6) ;
\draw [shift={(262.5,231.6)}, rotate = 210.06] [color={rgb, 255:red, 0; green, 0; blue, 0 }  ][line width=0.75]    (10.93,-3.29) .. controls (6.95,-1.4) and (3.31,-0.3) .. (0,0) .. controls (3.31,0.3) and (6.95,1.4) .. (10.93,3.29)   ;
\draw    (150,80) -- (150,42) ;
\draw [shift={(150,40)}, rotate = 90] [color={rgb, 255:red, 0; green, 0; blue, 0 }  ][line width=0.75]    (10.93,-3.29) .. controls (6.95,-1.4) and (3.31,-0.3) .. (0,0) .. controls (3.31,0.3) and (6.95,1.4) .. (10.93,3.29)   ;
\draw    (75,209.9) -- (38.93,230.6) ;
\draw [shift={(37.2,231.6)}, rotate = 330.14] [color={rgb, 255:red, 0; green, 0; blue, 0 }  ][line width=0.75]    (10.93,-3.29) .. controls (6.95,-1.4) and (3.31,-0.3) .. (0,0) .. controls (3.31,0.3) and (6.95,1.4) .. (10.93,3.29)   ;
\draw   (336.2,166.6) .. controls (336.2,118.77) and (374.97,80) .. (422.8,80) .. controls (470.63,80) and (509.4,118.77) .. (509.4,166.6) .. controls (509.4,214.43) and (470.63,253.2) .. (422.8,253.2) .. controls (374.97,253.2) and (336.2,214.43) .. (336.2,166.6) -- cycle ;
\draw  [fill={rgb, 255:red, 0; green, 0; blue, 0 }  ,fill opacity=1 ] (420.3,80) .. controls (420.3,78.62) and (421.42,77.5) .. (422.8,77.5) .. controls (424.18,77.5) and (425.3,78.62) .. (425.3,80) .. controls (425.3,81.38) and (424.18,82.5) .. (422.8,82.5) .. controls (421.42,82.5) and (420.3,81.38) .. (420.3,80) -- cycle ;
\draw  [fill={rgb, 255:red, 0; green, 0; blue, 0 }  ,fill opacity=1 ] (495.3,209.9) .. controls (495.3,208.52) and (496.42,207.4) .. (497.8,207.4) .. controls (499.18,207.4) and (500.3,208.52) .. (500.3,209.9) .. controls (500.3,211.28) and (499.18,212.4) .. (497.8,212.4) .. controls (496.42,212.4) and (495.3,211.28) .. (495.3,209.9) -- cycle ;
\draw  [fill={rgb, 255:red, 0; green, 0; blue, 0 }  ,fill opacity=1 ] (345.3,209.9) .. controls (345.3,208.52) and (346.42,207.4) .. (347.8,207.4) .. controls (349.18,207.4) and (350.3,208.52) .. (350.3,209.9) .. controls (350.3,211.28) and (349.18,212.4) .. (347.8,212.4) .. controls (346.42,212.4) and (345.3,211.28) .. (345.3,209.9) -- cycle ;
\draw    (497.8,209.9) -- (533.57,230.6) ;
\draw [shift={(535.3,231.6)}, rotate = 210.06] [color={rgb, 255:red, 0; green, 0; blue, 0 }  ][line width=0.75]    (10.93,-3.29) .. controls (6.95,-1.4) and (3.31,-0.3) .. (0,0) .. controls (3.31,0.3) and (6.95,1.4) .. (10.93,3.29)   ;
\draw  [fill={rgb, 255:red, 0; green, 0; blue, 0 }  ,fill opacity=1 ] (420.3,166.6) .. controls (420.3,165.22) and (421.42,164.1) .. (422.8,164.1) .. controls (424.18,164.1) and (425.3,165.22) .. (425.3,166.6) .. controls (425.3,167.98) and (424.18,169.1) .. (422.8,169.1) .. controls (421.42,169.1) and (420.3,167.98) .. (420.3,166.6) -- cycle ;
\draw    (422.8,80) -- (422.8,42) ;
\draw [shift={(422.8,40)}, rotate = 90] [color={rgb, 255:red, 0; green, 0; blue, 0 }  ][line width=0.75]    (10.93,-3.29) .. controls (6.95,-1.4) and (3.31,-0.3) .. (0,0) .. controls (3.31,0.3) and (6.95,1.4) .. (10.93,3.29)   ;
\draw    (347.8,209.9) -- (311.73,230.6) ;
\draw [shift={(310,231.6)}, rotate = 330.14] [color={rgb, 255:red, 0; green, 0; blue, 0 }  ][line width=0.75]    (10.93,-3.29) .. controls (6.95,-1.4) and (3.31,-0.3) .. (0,0) .. controls (3.31,0.3) and (6.95,1.4) .. (10.93,3.29)   ;
\draw [color={rgb, 255:red, 120; green, 120; blue, 120 }  ,draw opacity=1 ] [dash pattern={on 4.5pt off 4.5pt}]  (560,83) -- (560,247) ;
\draw [shift={(560,250)}, rotate = 270] [fill={rgb, 255:red, 120; green, 120; blue, 120 }  ,fill opacity=1 ][line width=0.08]  [draw opacity=0] (8.93,-4.29) -- (0,0) -- (8.93,4.29) -- cycle    ;
\draw [shift={(560,80)}, rotate = 90] [fill={rgb, 255:red, 120; green, 120; blue, 120 }  ,fill opacity=1 ][line width=0.08]  [draw opacity=0] (8.93,-4.29) -- (0,0) -- (8.93,4.29) -- cycle    ;
\draw  [fill={rgb, 255:red, 0; green, 0; blue, 0 }  ,fill opacity=1 ] (147.5,166.6) .. controls (147.5,165.22) and (148.62,164.1) .. (150,164.1) .. controls (151.38,164.1) and (152.5,165.22) .. (152.5,166.6) .. controls (152.5,167.98) and (151.38,169.1) .. (150,169.1) .. controls (148.62,169.1) and (147.5,167.98) .. (147.5,166.6) -- cycle ;
\draw [color={rgb, 255:red, 120; green, 120; blue, 120 }  ,draw opacity=1 ] [dash pattern={on 4.5pt off 4.5pt}]  (143.5,90.44) -- (82.41,200.2) ;
\draw [shift={(80.95,202.82)}, rotate = 299.1] [fill={rgb, 255:red, 120; green, 120; blue, 120 }  ,fill opacity=1 ][line width=0.08]  [draw opacity=0] (8.93,-4.29) -- (0,0) -- (8.93,4.29) -- cycle    ;
\draw [shift={(144.95,87.82)}, rotate = 119.1] [fill={rgb, 255:red, 120; green, 120; blue, 120 }  ,fill opacity=1 ][line width=0.08]  [draw opacity=0] (8.93,-4.29) -- (0,0) -- (8.93,4.29) -- cycle    ;

\draw (155,62) node [anchor=north west][inner sep=0.75pt]  [font=\small] [align=left] {$\displaystyle \left( 0,\sqrt{6} /3\right)$};
\draw (75,215) node [anchor=north west][inner sep=0.75pt]  [font=\small] [align=left] {$\displaystyle \left( -\sqrt{2} /3,0\right)$};
\draw (170,215) node [anchor=north west][inner sep=0.75pt]  [font=\small] [align=left] {$\displaystyle \left(\sqrt{2} /3,0\right)$};
\draw (569,152) node [anchor=north west][inner sep=0.75pt]  [font=\small,color={rgb, 255:red, 120; green, 120; blue, 120 }  ,opacity=1 ] [align=left] {$\displaystyle  \begin{array}{l}
D_{1} \ =\ 4\sqrt{6} /9\\
\ \ \ \ \ \ =\ D_{0} \cdot 2\sqrt{3} /3
\end{array}$};
\draw (155,162) node [anchor=north west][inner sep=0.75pt]  [font=\small] [align=left] {$\displaystyle \left( 0,\sqrt{6} /9\right)$};
\draw (\branchTriangCoord ,112) node [anchor=north west][inner sep=0.75pt]  [font=\small,color={rgb, 255:red, 120; green, 120; blue, 120 }  ,opacity=1 ] [align=left] {$\displaystyle D_{0} \ =\ 2\sqrt{2} /3$};

\end{tikzpicture}}
    \caption{Example of slackness in diameter constraint for the $2$-simplex (after projecting into $\R^2$ and translating).}
    \label{Fig:EquilateralTriangle}
\end{figure}

\begin{remark}[On the Special Case of Interpolating with Strongly Convex Sets] \label{Rem:SpecialCase_SCInterp}
If we take $\beta$ to $\infty$, that is, entirely relaxing the requirement of smoothness, then our interpolation conditions reduce to 
\begin{align}
    & \|z_i-\oa n_i - z_j \| \leq \oa \label{Eqn:SpecialConditions_SC1}\\
    & \| z_i - \oa n_i - x_k \| \leq \oa - \delta_k \label{Eqn:SpecialConditions_SC2}\\
    & \| z_i - z_j \| \leq \lambda D  \label{Eqn:SpecialConditions_SC3}\\
    & \| z_i - x_k  \| \leq \lambda D - \delta_k  \label{Eqn:SpecialConditions_SC4}\\
    & \|x_k - x_\el \| \leq \lambda D - \delta_k - \delta_\el \label{Eqn:SpecialConditions_SC5}
\end{align}
with no more dependence on auxiliary parameters $\{w_k\}_{k \in \K}$.
As an immediate corollary, we have the following result for strongly convex sets.
\begin{corollary} \label{Cor:Interp-SC}
        Let $S=\SSetInterp$. If $S$ is $\mathcal{C}_{\alpha,\infty,D}$-interpolable, then for all $i,j \in \I$ and  $k,\el \in \K$, (\ref{Eqn:SpecialConditions_SC1}-\ref{Eqn:SpecialConditions_SC5}) are satisfied for $\lambda=1$. Conversely, if for all $i,j \in \I$ and  $k,\el \in \K$, (\ref{Eqn:SpecialConditions_SC1}-\ref{Eqn:SpecialConditions_SC5}) are satisfied for $\lambda = 1/\sqrt{2}$, then $S$ is $\mathcal{C}_{\alpha,\infty,D}$-interpolable.
\end{corollary}
\end{remark}

\begin{remark}[On the Special Case of Interpolating with \edit{Smooth, Convex Sets}] \label{Rem:SpecialCase_SmoothInterp}
    We can similarly relax the strong convexity condition, taking $\alpha$ to zero and deriving simpler conditions for smooth convex interpolation from \cref{Thm:Interpolation_NoSlack}. Suppose that $\|z_i - \oa n_i - (z_j - \ob n_j)\| \leq \og$. Then
    \begin{align*}
        & \langle z_i - \oa n_i - z_j + \ob n_j, z_i - \oa n_i - z_j + \ob n_j \rangle \leq \frac{1}{\gamma^2} \\
        \Leftrightarrow \quad & \| z_i - \ob n_i - z_j + \ob n_j\|^2 - 2 \langle \og n_i, z_i - \ob n_i - z_j + \ob n_j \rangle + \frac{1}{\gamma^2} \leq \frac{1}{\gamma^2} \\
        \Leftrightarrow \quad &  \langle n_i, z_j - \ob n_j - z_i + \ob n_i \rangle \leq -\frac{\gamma}{2} \|z_i - \ob n_i - z_j + \ob n_j\|^2 \ .
    \end{align*}
    Taking the limit as $\alpha \to 0$ and correspondingly $\gamma \to 0$, this inequality becomes
    \begin{equation} \label{Eqn:SpecialCondition_Smooth1}
        \langle n_i, z_j - \ob n_j - z_i + \ob n_i \rangle \leq 0
    \end{equation}
    in place of \eqref{Eqn:Condition1}.
    Similarly, for $w_k$, we obtain
    \begin{equation} \label{Eqn:SpecialCondition_Smooth2}
        \langle n_i, w_k + s_k n_i - z_i + \ob n_i \rangle \leq 0
    \end{equation}
    in place of \eqref{Eqn:Condition2}. We can then state a corollary for smooth sets.
    \begin{corollary}\label{Cor:Interp-Smooth}
        Let $S=\SSetInterp$. $S$ is $\mathcal{C}_{0,\beta,D}$-interpolable if and only if there exist $\{w_k\}_{k \in \K} \subset \R^d$ such that for all $i,j \in \I$ and $k,\el \in \K$, \eqref{Eqn:SpecialCondition_Smooth1} and \eqref{Eqn:SpecialCondition_Smooth2} along with \eqref{Eqn:Condition3}-\eqref{Eqn:Condition6} are satisfied for $\lambda=1$.
    \end{corollary}
    \noindent Taking the limit as $\beta$ tends to $\infty$ 
    \edit{yields a final result for nonsmooth, non-strongly convex interpolation.
    \begin{corollary}\label{Cor:ConvexInterp}
        A set $S=\SSetInterp$ is $\C_{0,\infty,D}$-interpolable if and only if for all $i,j \in \I$ and  $k,\el \in \K$,
        \begin{align*}
            & \langle n_i, z_j - z_i \rangle \leq 0\\
            & \langle n_i, x_k + \delta_k n_i - z_i \rangle \leq 0 \\
            & \|x_k - x_\el \| \leq D - \delta_k - \delta_\el \\
            & \|z_i - x_k \| \leq D - \delta_k \\
            & \|z_i - z_j \| \leq D  
        \end{align*}
        where $n_i = \unitv{i}{}$ for all $i \in \I$.
    \end{corollary}
    }
    
\end{remark}

Finally, we note the strong connections between our set interpolation and the function interpolation of \cite{Interpolation2}. In the convex case of \cref{Cor:ConvexInterp} above, if $\delta_k = 0$ for all $k$ (or if $\K = \emptyset)$,  
we can replace all uses of the normalized vectors $n_i$ with $v_i$ and recover an equivalent reformulation of \cite[Theorem 3.6]{Interpolation2}.
More generally, we show below that smooth function interpolation \cite{Interpolation} is recovered as a limit of our set interpolation theory. We defer the proof to \cref{App:Interp_SetFuncInterp}.

\begin{theorem} \label{Thm:SetFuncInterp_Smooth}
    The set of observations $\{(x_i,g_i, f_i)\}_{i \in \I}$ is $\mathcal{F}_{0,L}$-interpolable if and only if \\
    $\setInterpRescaledArgs$ is $\mathcal{C}_{0, \eta^2L, \infty}$-interpolable for all $\eta > 0$.
\end{theorem}

\subsection{Proof of Theorem~\ref{Thm:Interpolation_Necc}}
     Suppose that $S$ is $\C_{\alpha,\beta,D}$-interpolable. Then there exists some $C \in \C_{\alpha,\beta,D}$ such that $v_i \in N_C(z_i)$, $z_i \in C$, and $x_k \in \Interior{\delta_k} C$ for all $i \in \I, k \in \K$. From \cref{Lem:SmoothSCParam}, we know that $C = C_0 + B(0,\ob)$ for some $\gamma$-strongly convex set $C_0$, where $\og + \ob = \oa$. In particular, as shown in \cite{LiuGrimmer}, we can write $C_0$ explicitly with the Minkowski difference $C_0 = C - B(0,\ob) := \{x \mid x + B(0,\ob) \subseteq C\}$.

    As a first intermediate result, we claim for each $i\in\I$,
    \begin{equation}\label{Eqn:vi-inner-normality}
        v_i \in N_{C_0}(z_i - \unitv{i}{\beta}) \ .
    \end{equation}
    This follows since $v_i \in N_{C}(z_i)$ implies for all $y \in C$, $\langle v_i, y - z_i \rangle \leq 0$. So noting that every $p \in C_0$ has $p + \unitv{i}{\beta} \in C$, we have
    $\langle v_i, p - (z_i - \unitv{i}{\beta}) \rangle \leq 0$ for all $p\in C_0$. From the definition of $C_0$, $z_i - \unitv{i}{\beta} \in C_0$. Together these yield~\eqref{Eqn:vi-inner-normality}.
    
    Since $C_0$ is $\gamma$-strongly convex, \eqref{Eqn:vi-inner-normality} implies that $C_0 \subseteq B(z_i - \unitv{i}{\beta} - \unitv{i}{\gamma}, \og)$. Since $z_j - \unitv{j}{\beta} \in C_0$ for all $j \in\I$, we have $z_j - \unitv{j}{\beta} \in B(z_i - \unitv{i}{\beta}-\unitv{i}{\gamma}, \og)$. So \eqref{Eqn:Condition1} holds as
    \begin{equation*}
       \| z_i - \unitv{i}{\alpha} - (z_j - \unitv{j}{\beta})\| = \| z_i - \unitv{i}{\beta} - \unitv{i}{\gamma} - (z_j - \unitv{j}{\beta}) \| \leq \og \ . 
    \end{equation*}

    Now, we verify \eqref{Eqn:Condition2} and \eqref{Eqn:Condition3} for each $k\in\K$ through two cases:
    
    \noindent {\bf Case 1: $x_k \in C_0$.} In this case, we let $w_k = x_k$. Suppose that $\delta_k > \ob$, so $s_k := \max \{0, \delta_k - \ob \} = \delta_k - \ob$. We claim that $x_k \in \Interior{s_k} C_0$. Suppose that for some $\zeta$, $x_k + (\delta_k - \ob)\unit{\zeta} \notin C_0$. Then $x_k + (\delta_k - \ob) \unit{\zeta} + \ob \unit{\zeta} = x_k + \delta_k \unit{\zeta} \notin C$. But since $x_k \in \Interior{\delta_k} C$, this is a contradiction, so we must have $x_k \in \Interior{s_k} C_0$. Now suppose $\delta_k \leq \ob$. Then $s_k = 0$, so $x_k \in \Interior{s_k} C_0$ is true by assumption. Since $x_k = w_k$, we have shown that $w_k + B(0,s_k) \subseteq C_0$. When combined with the previous fact that $C_0 \subseteq B(z_i - \unitv{i}{\beta} - \unitv{i}{\gamma}, \og)$, it follows that
    \begin{equation*}
        \|z_i - \unitv{i}{\alpha} - w_k \| \leq \og - s_k \ .
    \end{equation*}
    This is equivalent to \eqref{Eqn:Condition2}.
    Finally, noting $\| w_k - x_k \| = 0 \leq \max\{0, \ob - \delta_k \} = \ob - \delta_k + s_k$, we conclude \eqref{Eqn:Condition3} is satisfied.

    \noindent {\bf Case 2: $x_k \notin C_0$.} Since $C = C_0 + B(0,\ob)$ and $x_k \notin C_0$, we must have $\delta_k < \ob$. We therefore have $s_k=0$, and we set $w_k$ as the orthogonal projection of $x_k$ onto $C_0$ and $\zeta = \frac{x_k-w_k}{\|x_k-w_k\|}$ as the corresponding normal vector to $C_0$. Since $w_k\in C_0$, the same reasoning as above leveraging~\eqref{Eqn:vi-inner-normality} implies \eqref{Eqn:Condition2}. We know $x_k \in \Interior{\delta_k} C$, so it follows that $w_k + \zeta(\|w_k-x_k\|+\delta_k) \in C$. However, since $\zeta\in N_{C_0}(w_k)$, it follows that $\|w_k-x_k\|+\delta_k \leq \ob$. Recalling $s_k=0$, that is exactly \eqref{Eqn:Condition3}.
    
    Finally, we verify the diameter conditions \eqref{Eqn:Condition4}, \eqref{Eqn:Condition5}, and \eqref{Eqn:Condition6}. For all $i \in \I, k \in \K$, we have shown $z_i - \unitv{i}{\beta} \in C_0$ and in either case above, we have that $w_k \in \Interior{s_k} C_0$. Hence the diameter bound of $\diam(C_0) \leq D - \frac{2}{\beta}$ from \cref{Lem:SmoothDiam} yields the remaining three conditions for $\lambda = 1$.

\subsection{Proof of Theorem~\ref{Thm:Interpolation_SuffWithSlack}}
    Suppose our \edit{six conditions} hold for $\lambda = \sqrt{\frac{d+1}{2d}}\geq \frac{1}{\sqrt{2}}$. We begin by constructing our interpolating set $C$. First, we construct $C_0$ as the $\gamma$-strongly convex hull of $\{B(w_k, s_k) \}_{k \in \K}$ and $\{z_i - \unitv{i}{\beta} \}_{i \in \I}$. More formally, by \cite[Proposition 2.5]{Vial_StrongAndWeak}, we set $C_0 = \bigcap_{y\in Y} B(y,\og)$, where 
    \begin{equation*}
        Y = \{ y \mid \|z_i - \unitv{i}{\beta} - y \| \leq \og \quad \forall i \in \I, \quad \|w_k - y\| \leq \og - s_k  \quad \forall k \in \K \} \ .
    \end{equation*}
    Observe that each ball $B(y,\og)$ is $\gamma$-strongly convex, and this property is preserved under intersections~\cite[Proposition 2]{Vial_StrongConvexity}, so $C_0$ is $\gamma$-strongly convex. We then define $C = C_0 + B(0,\ob)$. By \cref{Lem:SmoothSCParam}, we see that $C$ is $\alpha$-strongly convex and $\beta$-smooth. 
    All that remains is to verify this set $C$ correctly interpolates the given observation data $S$ and has the desired diameter bound $D$.

    \noindent {\bf Verification that $z_i\in C$ and $v_i \in N_C(z_i)$ for each $i\in\I$.}
    From the definition of $Y$, $z_i - \unitv{i}{\beta} \in B(y,\og)$ for all $y \in Y$. Hence $z_i - \unitv{i}{\beta} \in C_0$. Since $C=C_0+B(0,\ob)$, it follows that $z_i \in C$. Further, by \eqref{Eqn:Condition1}, $\|z_i - \unitv{i}{\alpha} - (z_j - \unitv{j}{\beta}) \| \leq \og$ for all $j\in\I$, and by \eqref{Eqn:Condition2}, $\|z_i - \unitv{i}{\alpha} - w_k\|\leq \og-s_k$ for all $k\in\K$. Hence $z_i - \unitv{i}{\alpha} \in Y$. Letting $B_i = B(z_i - \unitv{i}{\alpha}, \og)$, observe that $v_i \in N_{B_i}(z_i - \unitv{i}{\alpha} + \unitv{i}{\gamma}) = N_{B_i}(z_i - \unitv{i}{\beta})$. Combining this with the fact that $C_0 \subseteq B_i$ since $z_i - \unitv{i}{\alpha} \in Y$, one has that $v_i \in N_{C_0}(z_i - \unitv{i}{\beta})$. Finally, since $C = C_0 + B(0,\ob)$, it follows that $v_i \in N_C(z_i)$ as well.

    \noindent {\bf Verification that $x_k\in\Interior{\delta_k} C$ for each $k\in\K$.}
    Next, we consider $w_k$ and $x_k$. By definition of $Y$, we see that for all $y \in Y$, $w_k \in \Interior{s_k} B(y, \og)$. Consequently, by construction, $w_k \in \Interior{s_k} C_0$. Then using the fact that $\|x_k - w_k \| \leq \ob + s_k - \delta_k$ by \eqref{Eqn:Condition3}, for any $\zeta$, it follows that
    \begin{align*}
        \|x_k + \delta_k \unit{\zeta} - (w_k + s_k \unit{\zeta}) \| & \leq \|x_k - w_k \| + |\delta_k - s_k | = \|x_k - w_k \| + \delta_k - s_k \\
        & \leq \ob + s_k - \delta_k + (\delta_k - s_k) = \ob \ .
    \end{align*}
    Since $w_k \in \Interior{s_k} C_0$, for any nonzero $\zeta$, one has $\hat{w}_k := w_k + s_k \unit{\zeta} \in C_0$. Hence $\|x_k + \delta_k \unit{\zeta} - \hat{w}_k \| \leq \ob$ from which one can conclude for any $\zeta$, $x_k + \delta_k \unit{\zeta} \in C$ and consequently $x_k \in \Interior{\delta_k} C$. 

    \noindent {\bf Verification of diameter bound $\diam(C) \leq D$.}
    By \cref{Lem:SmoothDiam}, it is sufficient to show that $\diam(C_0) \leq D - \frac{2}{\beta}$. Since $C_0$ is $\gamma$-strongly convex, we already know $\diam(C_0) \leq \frac{2}{\gamma}$. So, all that remains is to prove the needed bound when $D - \frac{2}{\beta} < \frac{2}{\gamma}$.
    Jung's theorem~\cite{Jung1901} states that given a compact set $X \subseteq \R^d$, there exists a closed ball $B$ with radius $R = \diam(X)\sqrt{\frac{d}{2(d+1)}}$ containing $X$.
    Applied to the set $X = \{z_i - \unitv{i}{\beta}\}_{i\in\I} \ \cup \ \{B(w_k, s_k) \}_{k \in \K}$, which has diameter at most $\sqrt{\frac{d+1}{2d}}(D - \frac{2}{\beta})$ by~\eqref{Eqn:Condition4}, \eqref{Eqn:Condition5}, and~\eqref{Eqn:Condition6}, there must exist $q\in\mathbb{R}^d$ such that $X\subseteq B(q,R)$ where $R=\left(\sqrt{\frac{d+1}{2d}}(D - \frac{2}{\beta})\right)\sqrt{\frac{d}{2(d+1)}} = \frac{1}{2}(D-\frac{2}{\beta}) $.
    Hence the distance from $q$ to any point in $X$ is at most $ \frac{1}{2}(D-\frac{2}{\beta}) < 1/\gamma$. As a result, a neighborhood of $q$ lies in $Y$. Namely,
    $ B(q, \og -  \frac{1}{2}(D-\frac{2}{\beta})) \subseteq Y$. So,
    $$ C_0 = \bigcap_{y\in Y} B\left(y,\og\right) \subseteq \bigcap_{y\in B(q, \og -  \frac{1}{2}(D-\frac{2}{\beta}))} B\left(y,\og\right) = B\left(q, \frac{1}{2}(D-\frac{2}{\beta})\right) \ , $$
    proving the needed bound $\diam(C_0) \leq D - \frac{2}{\beta}$.

\subsection{Proof of Theorem~\ref{Thm:Interpolation_NoSlack}}
The forward direction is already proven by \cref{Thm:Interpolation_Necc}, so we only need to prove the reverse. 
First consider the case of $D \geq \frac{2}{\alpha}$. By our argument in \cref{Thm:Interpolation_SuffWithSlack}, if \edit{our conditions} hold for $\lambda=1$, then there exists an $\alpha$-strongly convex, $\beta$-smooth set $C$ (of unspecified diameter) that interpolates $S$. However, since any $\alpha$-strongly convex set $C$ must satisfy $C \subseteq B(z - \oa n, \oa)$ for $z \in \bdry C$ with unit normal vector $n$, we see that $\diam(C) \leq \diam(B(z - \oa n, \oa)) = \frac{2}{\alpha} \leq D$. Therefore, $S$ is $\C_{\alpha,\beta,D}$-interpolable.

Now consider the case of $\alpha=0$. Suppose our \edit{six conditions} hold for $\lambda = 1$. As shown in \cref{Rem:SpecialCase_SmoothInterp}, our conditions \eqref{Eqn:Condition1} and \eqref{Eqn:Condition2} reduce to \eqref{Eqn:SpecialCondition_Smooth1} and \eqref{Eqn:SpecialCondition_Smooth2}. Define $C_0 = \textrm{conv} (\{z_i - \unitv{i}{\beta} \}_{i \in \I}, \{ B(w_k,s_k) \}_{k \in \K} )$. Clearly $C_0$ is convex with $w_k \in \Interior{s_k} C_0$ and $z_i - \unitv{i}{\beta} \in C_0$. We then construct $C = C_0 + B(0,\ob)$. From \cref{Prop:SmoothProperties}, $C$ is $\beta$-smooth. 

Next, we show that $z_i\in C$ and $v_i \in N_C(z_i)$. Since $z_i - \unitv{i}{\beta} \in C_0$, we immediately have that $z_i \in C$. Consider any $y \in C$ and let $p\in C_0$ be such that $y = p + \ob \zeta$ for some $\|\zeta\|\leq 1$. Since $C_0$ is a convex hull, we write $p$ as the convex combination 
\begin{equation*}
    p = \sum_{j \in \I} \sigma_j (z_j -  \unitv{j}{\beta}) + \sum_{k \in \K} \phi_k ( w_k + s_k \xi_k)
\end{equation*}
where $\sigma_j, \phi_k \geq 0$ and $\sum_j \sigma_j + \sum_k \phi_k = 1$, and $\|\xi_k \|\leq 1$.
We then have
\begin{align*}
    \langle v_i, & y - z_i \rangle  = \langle v_i, p + \ob \zeta - z_i \rangle \leq \langle v_i, p - (z_i - \unitv{i}{\beta}) \rangle \\
    &  = \sum_{j \in \I} \sigma_j \langle v_i, z_j - \unitv{j}{\beta} - (z_i - \unitv{i}{\beta}) \rangle + \sum_{k \in \K} \phi_k \langle v_i, w_k + s_k \xi_k - (z_i - \unitv{i}{\beta}) \rangle \\
    & \leq \sum_{j \in \I} \sigma_j \langle v_i, z_j - \unitv{j}{\beta} - (z_i - \unitv{i}{\beta}) \rangle + \sum_{k \in \K} \phi_k \langle v_i, w_k + s_k \unit{v_i} - (z_i - \unitv{i}{\beta}) \rangle \\
    & \leq 0
\end{align*}
where the last inequality follows from \eqref{Eqn:SpecialCondition_Smooth1} and \eqref{Eqn:SpecialCondition_Smooth2}.
Hence $v_i \in N_C(z_i)$.

By the same argument as in \cref{Thm:Interpolation_SuffWithSlack}, we have that $x_k \in \Interior{\delta_k} C$. Lastly, since $C_0$ is a convex hull, \eqref{Eqn:Condition4}, \eqref{Eqn:Condition5}, and \eqref{Eqn:Condition6} establish that
\begin{equation*}
    \diam(C_0) = \diam \left( \{z_i - \unitv{i}{\beta}\}_{i \in \I} \  \cup \  \{B(w_k,s_k)\}_{k \in \K} \right) \leq D - \frac{2}{\beta} \ . 
\end{equation*}
Then, by \cref{Lem:SmoothDiam}, we conclude that $\diam(C) \leq D$.

    \section{Separating Hyperplane Algorithms as an SDP}\label{Sec:SepHyperplane}
In the remainder of this paper, we provide \edit{three} applications of our interpolation theorems. In this section, we consider a family of simple algorithms computing an element of a convex set by iteratively using a separating hyperplane oracle. Our interpolation theorems allow us to quantify the worst-case stopping time of such a method over any smooth and/or strongly convex set with \edit{an interior point} via semidefinite programming. Motivated by and confirming our numerical results, we identify a simple minimax optimal separating hyperplane algorithm.

For any closed convex set $C \subseteq \R^d$, we denote the set of (unit) separating hyperplanes of $C$ at some $x\in\R^d$ by
\begin{equation*}
    SH_C(x) = \{ n\in\R^d \mid \|n\|=1,\ \langle n, y-x \rangle \leq 0 \quad \forall y\in C\} \ .
\end{equation*}
Note this set is nonempty exactly when $x\notin \interior C$. Here, our primary interest is in designing algorithms constructing a member of $\interior C$ using a sequence of separating hyperplane oracle queries. \edit{For example, this generalizes convex minimization, seeking an $\epsilon$-minimizer where each subgradient provides such a separating hyperplane to the $\epsilon$-level set.} As a general form of method with $N$ queries, consider any iteration producing points $x_i$ via
\begin{equation} \label{Eqn:SH-general-step}
    x_{i+1} = x_i - \sum_{j=0}^i H_{i,j} n_j, \qquad n_i\in SH_C(x_i) \ , \qquad \forall i=0,\dots,N-1
\end{equation}
parameterized by the lower triangular matrix of stepsizes $H$. Note this method must halt once $x_{i}\in\interior C$ as $SH_C(x_i)$ is empty.

Given fixed $\alpha,\beta,\delta,R$, we consider the problem finding a feasible point to any set $C \in \C_{\alpha,\beta,\infty}$ containing some $q\in\Interior{\delta} C$ and any initialization $x_0$ with $\|x_0-q\|\leq R$. Since no diameter bound is enforced on $C$, our interpolation conditions are necessary and sufficient (i.e., \cref{Thm:Interpolation_NoSlack} applies).

The question of whether a problem instance exists where an algorithm $H$ can fail to construct some $x_i\in \interior C$ with $i=0,\dots, N$ corresponds to
\begin{equation*}
    \exists x_i, n_i, q, C \qquad \mathrm{s.t.}\qquad \begin{cases} 
        & C \in \C_{\alpha,\beta,\infty} \\
        & B(q,\delta) \subseteq C \\
        & n_i\in SH_C(x_i) \\
        & \|x_0 - q\| \leq R \\
        & x_{i+1} = x_i - \sum_{j=0}^i H_{i,j} n_j \ .
    \end{cases}
\end{equation*}
A successful algorithm design would be a selection of $H$ such that no solution to the above system exists, as a failure to have $n_k\in SH_C(x_k)$ exist for some $k$ implies $x_k\in\interior C$. Noting that every separating hyperplane $n_i$ must be normal to the set $C$ at some $z_i$, this system can be rewritten in terms of interpolation with $\I=\{0,\dots,N\}$. \cref{Thm:Interpolation_NoSlack} enables this decision problem to be described as a system of quadratic inequalities using our first three interpolation conditions~\eqref{Eqn:Condition1}-\eqref{Eqn:Condition3}:
\begin{equation}
    \exists x_i, z_i, n_i, q, w \qquad \mathrm{s.t.}\qquad \begin{cases} 
        & \|z_i-\oa n_i - (z_j - \ob n_j)  \|^2 \leq \frac{1}{\gamma^2} \\
        & \| z_i - \oa n_i - w \|^2 \leq (\og - s)^2 \\
        & \|q - w \|^2 \leq (\ob - \delta + s)^2 \\ 
        & \langle n_i, z_i - x_i\rangle \leq 0\\
        & \|n_i\|^2=1 \\
        & \|x_0 - q\|^2 \leq R^2 \\
        & x_{i+1} = x_i - \sum_{j=0}^i H_{i,j} n_j
    \end{cases}\label{Eqn:OptSH-QCQP}
\end{equation}
where $\gamma = (\oa - \ob)^{-1}$ and $s = \max\{0, \delta - \ob \}$.
Without loss of generality, we fix $q = 0$.
Assuming $d\geq 2N+4$, a Gram matrix reformulation of this quadratic system yields an equivalent SDP feasibility problem with variable $G$ capturing every quadratic term in the original variables $x_0,z_i,n_i,w$, i.e.,
\begin{align*}
        & \gramMatrix = [x_0|z_0|z_1|\dots|z_N|n_0|n_1|\dots|n_N|w] \in \R^{d\times(2N+4)} \ ,\\
        & G = \gramMatrix^T \gramMatrix \in \mathbb{S}_+^{2N+4} \ . 
    \end{align*}
\edit{In the case $d < 2N+4$, this reformulation still represents a valid upper bound on our performance estimation problem, but it is no longer guaranteed to be tight (See \cref{Rem:DimensionBound}).} \cref{App:SH-SDP-reformulation} presents this reformulation in full for the sake of completeness. Such reformulations are widespread in the existing PEP literature. Hence, we can efficiently certify whether a proposed separating hyperplane algorithm is guaranteed to construct a feasible point.

Consider the simple separating hyperplane method fixing $H$ to be diagonal \edit{with constant value}:
\begin{equation} \label{Eqn:SH-constant-step}
    x_{i+1} = x_i - h n_i \ , \qquad n_i\in SH_C(x_i)  \ , \qquad \forall i=0,\dots,N-1 \ .
\end{equation}
For a given family of problem instances fixing $\alpha,\beta,\delta,R$, we can compute the number of iterations needed to guarantee a strictly feasible point by growing $N$ until the corresponding SDP becomes infeasible. Denote this maximal number of steps needed to ensure an interior point is found by $N_{\mathrm{max}}$. \edit{Fixing $h=\max\{\delta,\ob\}$,} \cref{Fig:SepHMaxNData} shows $N_{\mathrm{max}}$ as $\delta$ and $R$ vary under this stepsize. From these numerical results, one can readily identify a formula $N_{\mathrm{max}} = \lfloor \frac{(R+h-\delta)^2}{h^2} \rfloor$.

\begin{figure}
    \centering
    \includegraphics[width=0.95\textwidth]{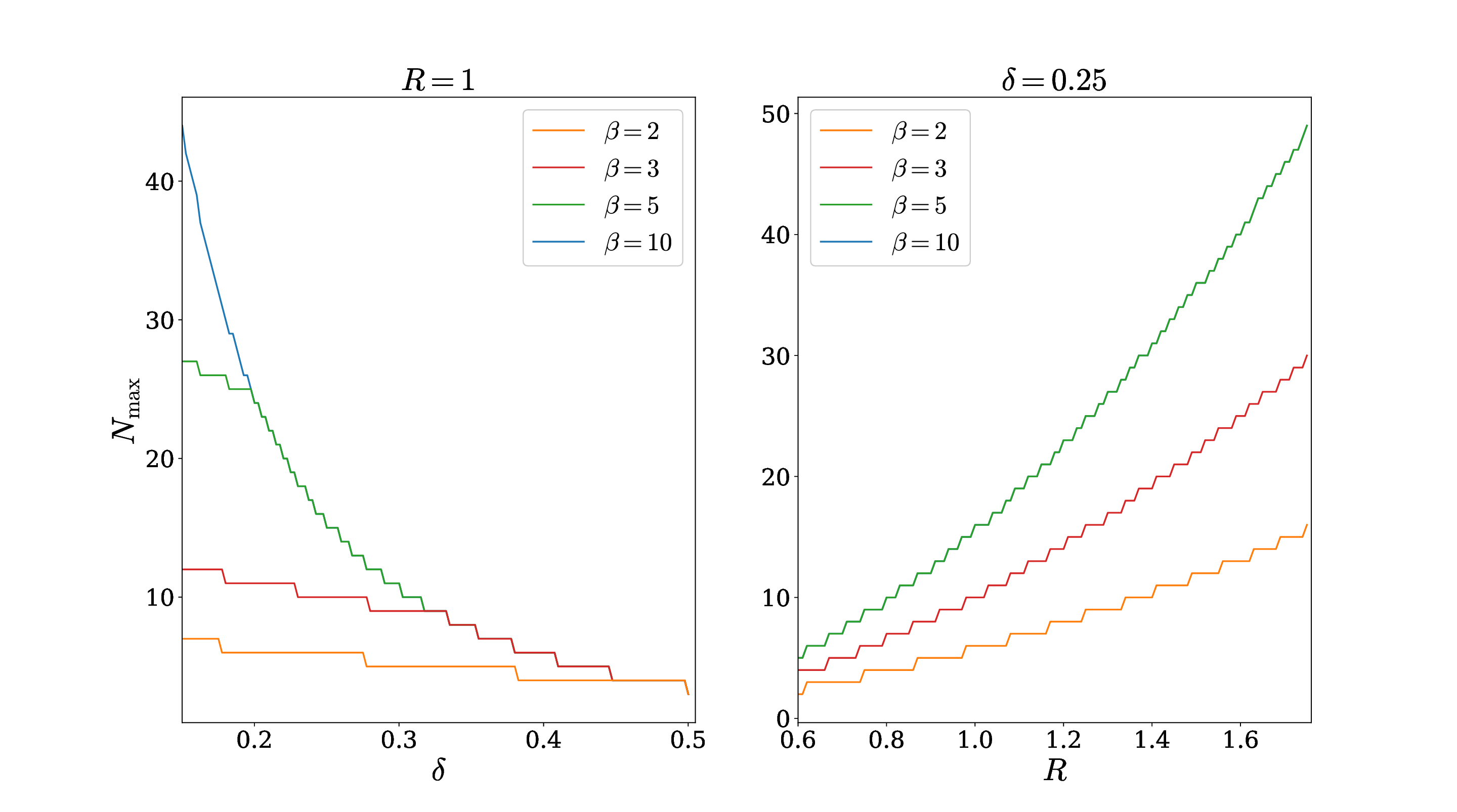}
    \caption{Numerical PEP result for $N_{\mathrm{max}}$ given a constant stepsize $h = \max \{\delta, \ob \}$ as $\delta$ varies on the left and $R$ on the right. In the second plot, the line for $\beta=10$ is covered by $\beta=5$.}
    \label{Fig:SepHMaxNData}
\end{figure}

This simple method and numerically observed rate turn out to be the minimax optimal method and guarantee for any separating hyperplane method of the general form~\eqref{Eqn:SH-general-step}. \edit{\cref{App:SH-minimax-optimal-proof} presents an elementary proof of this fact.}
\begin{theorem}\label{Thm:SH-minimax-optimal}
    For any $C \in \mathcal{C}_{\alpha,\beta,D}$ with $q \in \Interior{\delta} C$ and $\|x_0 - q \| \leq R$, the iteration~\eqref{Eqn:SH-constant-step} with stepsize $h=\max\{\delta,\ob\}$ must halt with $x_i \in \interior C$ by iteration $i \leq N_{\mathrm{max}}:= \lfloor \frac{(R+h-\delta)^2}{h^2} \rfloor$.
    Moreover, there exists $C,q,x_0$ \edit{and separating hyperplane oracle} such that for any method~\eqref{Eqn:SH-general-step} has $x_0,\dots,x_{N_{\mathrm{max}}-1}\notin \interior C$.
\end{theorem}
In this case, the proof of this minimax optimality theorem is sufficiently simple that one could have reached these conclusions without computer-assistance. \edit{Nevertheless, these results demonstrate the new possibilities of our set interpolation theory}. The numerical insights generated in the following two sections escalate in complexity beyond what is reasonable to do ``by-hand''.

    \section{Frank-Wolfe as an SDP with Nonconvex Constraints}\label{Sec:FW}
In this section, we consider performance estimation of the Frank-Wolfe (conditional gradient) method on smooth and strongly convex sets, \edit{extending the prior analysis of Taylor, Hendrickx, and Glineur~\cite{Interpolation2} in the case of convex sets}. 
This method applies to constrained optimization problems
\begin{equation}\label{Eqn:ConstrainedOpt}
        \min_{x \in C} f(x)
\end{equation}
where $f$ is at least convex and $L$-smooth and $C$ is at least a convex set with diameter at most $D$. At various points, we will consider problem classes with additional structures like $\mu$-strong convexity of $f$ and $\alpha$-strong convexity and $\beta$-smoothness of $C$. We will let $x_\star \in C$ denote a minimizer of~\eqref{Eqn:ConstrainedOpt} which exists by compactness.

\edit{Given a} gradient oracle for $f$ and a linear minimization oracle for $C$, the Frank-Wolfe method generates iterates $x_1, \dots, x_N\in C$ via
\begin{align}\label{Eqn:FrankWolfe}
    & z_k \in \argmin_{y \in C} \langle \nabla f(x_k), y \rangle \tag{FW} \\
    & x_{k+1} = (1-h_k)x_k + h_k z_k \nonumber
\end{align}
for some stepsize schedule $h = (h_0, \dots, h_{N-1})$. Such stepsizes are sometimes called ``open-loop''. Typically, stepsizes follow $h_k = \frac{2}{k+2}$ or more generally $h_k = \frac{\ell}{k+\ell}$ for some $\ell > 0$. In all cases, we will assume that $0 \leq h_k \leq 1$ for all $k$, ensuring $x_{k+1}$ remains a convex combination of $x_0$ and all $z_i$. 

\subsection{Frank-Wolfe PEP Formulation and Computational Approaches} \label{Subsec:FW-Derive-PEP}

For a fixed algorithm, determined by fixing \edit{$h$} above, its worst-case performance after $N$ steps can be formulated as a performance estimation \edit{problem.}
\edit{The} worst-case \edit{minimum} objective gap seen is given by
\begin{align} \label{Eqn:OptFWInfinite}
    p_\mathrm{FW}(N,h;\mu, L, \alpha, \beta, D) = \begin{cases}
        \max\limits_{f,C,x_0} & \edit{ \min\limits_{k=0,\dots, N}f(x_k) - f(x_\star) } \\
        \text{s.t.} \quad & \edit{x_{i+1} = (1-h_i)x_i + h_i z_i} \\
        & \edit{z_i \in \argmin_{i \in C} \langle \nabla f(x_i), y \rangle} \\
        & x_0 \in C \\
        & C \in \C_{\alpha,\beta,D} \\
        & f \in \mathcal{F}_{\mu,L} \\
        & x_\star \in \argmin_{y \in C} f(y) \ . 
    \end{cases}
\end{align}
\edit{From first-order optimality, $z_i \in \argmin_{y \in C} \langle \nabla f(x_i), y \rangle$ can be replaced by the equivalent condition $-\nabla f(x_i) \in N_C(z_i)$. Then our interpolation theory provides a reformulation as a solvable finite problem. See \cref{App:FW-Extra} for further details supporting this reformulation. For ease of exposition, in what follows}, we consider the specific case of optimizing $L$-smooth functions over $\beta$-smooth \edit{sets}, but analogous derivations follow for any combination of smoothness and strong convexity of functions and sets. \edit{We will separately consider the cases \edit{$\nabla f(x_\star)= 0$ and $\nabla f(x_\star) \neq 0$, here focusing on $\nabla f(x_\star) \neq 0$}. This separation is further motivated by classic results in the literature \cite{Wolfe,Guelat1986,garber2015faster,Wirth_FWAccel} where the behavior of the Frank-Wolfe algorithm depends significantly on whether $x_\star \in \interior C$. 
}

\edit{Using both set and function interpolation theory, we can reformulate~\eqref{Eqn:OptFWInfinite} as the following QCQP with additional separable nonconvex equality constraints.} \edit{Letting $\theta=(h,\mu,L,\alpha,\beta,D)$, with $\alpha=0$ for ease of presentation, }
\begin{equation}\label{Eqn:OptFWFinite}
    p_{\mathrm{FW}}(N;\theta) = \left\{ \begin{array}{lll} \max\limits_{\substack{x_k,g_k,f_k \\w_k,z_i,n_i}} \quad & \fmin - f_\star \\
        \text{s.t.} \quad & \fmin - f_k \leq 0 \quad & \forall k \in \K \\
        &  f_k - f_\el +\langle g_k,x_\el-x_k\rangle +\frac{1}{2L}\| g_k - g_\el\|^2 \leq 0 \quad & \forall k,\el \in \KStar \\
        & x_{i+1} = (1-h_i)x_i + h_i z_i \quad & \forall i \in \I \\
        & \langle -g_i, z_j-\frac{1}{\beta}n_j - (z_i -\frac{1}{\beta}n_i) \rangle \leq 0 \quad & \forall i,j \in \IStar \\
        & \langle -g_i, w_k - (z_i - \frac{1}{\beta}n_i) \rangle \leq 0 \quad & \forall i \in \IStar, k \in \KStar \\
        & \|x_k - w_k \|^2 \leq \obsq \quad & \forall k \in \KStar \\
        & \| z_i - \ob n_i - (z_j - \ob n_j) \|^2 \leq (D-\frac{2}{\beta})^2 \quad & \forall i,j \in \IStar \\
        & \|z_i - \ob n_i - w_k \|^2 \leq (D-\frac{2}{\beta})^2 \quad & \forall i \in \IStar, k \in \KStar \\
        & \|w_k - w_\el \|^2 \leq (D-\frac{2}{\beta})^2 \quad & \forall k,\el \in \KStar \\
        & \langle g_i, n_i \rangle \leq 0  \quad & \forall i \in \IStar \\
        & \| n_i \|^2 = 1 \quad & \forall i \in \IStar \\
        & \langle g_i, n_i \rangle^2 = \| g_i \|^2 \quad & \forall i \in \IStar \ .
    \end{array}\right.
\end{equation}

\noindent Note that the last three constraints \edit{above} are equivalent to the condition $n_i = \frac{-g_i}{\|g_i\|}$.
\edit{In the case that $g_\star = 0$, these constraints must be removed as $n_\star$ is undefined, requiring a branch in our formulation based on $g_\star$.}

\edit{One} can repeat the ``standard'' Gram matrix reformulation approach discussed in \cref{Sec:SepHyperplane}. Assuming $d \geq 4N+6$, one can derive an equivalent problem with variables $G$ and $F$ as
\begin{align*}
    & F = [f_0|f_1|\dots|f_N|\fmin] \in \R^{1 \times (N+2)} \\
    & \gramMatrix = [x_0|g_\star|g_0|g_1|\dots|g_N|z_0|\dots|z_{N-1}|n_\star|n_0|\dots|n_{N-1}|w_\star|w_0|\dots|w_N] \in \R^{d\times(4N+6)} \\
    & G = \gramMatrix^T \gramMatrix \in \mathbb{S}_+^{4N+6} \ .
\end{align*}
\edit{We omit the restatement of this program after this standard change of variables.} This Grammian change-of-variables reformulates all of the constraints of \eqref{Eqn:OptFWFinite} as linear constraints in our new variables, with the exception of the equality constraint $\langle g_i, n_i \rangle^2 = \| g_i \|^2$ which is quadratic in $G$. This constraint is nonconvex, and therefore cannot be expressed through a convex SDP.
\edit{We present three methods of solving or bounding the PEP \eqref{Eqn:OptFWFinite}, in order of increasing computational difficulty. In any case, we introduce the redundant constraints $\langle -g_i, n_j \rangle \leq \langle -g_i, n_i \rangle$ for all $i,j \in \IStar$, which prove useful in our relaxations below.}

\edit{As a simplest approach, one can relax the constraint $\langle g_i, n_i \rangle ^2 = \|g_i\|^2$ to $\langle g_i, n_i \rangle ^2 \leq \|g_i\|^2$, which is convex in $G$, providing a upper bound for \eqref{Eqn:OptFWInfinite}. This is similar to the original performance estimation approach of Drori and Teboulle \cite{FirstPEP}. While excluded from our numerics, this relaxation enables computations of performance upper bounds up to larger values of $N$ ($\sim 50$).} 

\edit{Alternatively, one can reformulate \eqref{Eqn:OptFWFinite} by introducing new variables $\psi_i$ to parameterize the nonconvex constraint.
Given $\psi = (\psi_\star, \psi_0, \dots, \psi_{N-1} ) \in \R^{N+1}_{\geq 0}$, one can define the parameterized problem $\pFWSubproblem(\psi;\theta)$ by adding the constraints $\|g_i\|^2 = \psi_i$ and $\langle -g_i, n_i \rangle = \sqrt{\psi_i}$, which are linear in $G$. Finding a local solution to \eqref{Eqn:OptFWFinite} then amounts to zeroth-order local maximization of 
$\pFWSubproblem(\psi; \theta)$ with respect to $\psi$, yielding a lower bound for \eqref{Eqn:OptFWFinite}.} This approach was computationally feasible to run up to $N\leq 15$.

\edit{Finally, complementing the above local computation of lower bounds, one can globally optimize~$\eqref{Eqn:OptFWFinite}$ via a simple branch-and-bound-type procedure. By Cauchy-Schwarz, $\langle -g_i, n_i \rangle \leq \|g_i\|$. Then one can apply piecewise McCormick relaxations \cite{McCormick_1976,PiecewiseMcCormick} to create a partition of convex relaxations of our equality condition $\langle -g_i, n_i \rangle  = \|g_i\|$. One can then globally optimize by iteratively refining this partition for each $i$. Given this procedure's computational cost scales exponentially in $N$, our computations of such global upper bounds were limited to $N\leq 5$. When tractable, these numerical upper bounds (denoted by dots in our figures) always closely aligned with our numerical lower bounds, together certifying tight bounds on the true optimal value.}

\subsection{\edit{Survey of Smooth/Strongly Convex Settings for Acceleration}} \label{Subsec:FWAccelSurvey}
\edit{Bounds on Frank-Wolfe's worst-case convergence rates are well-studied. For the most general setting of convex bounded $C$ and $L$-smooth convex $f$, the best known upper bound on convergence rates in this setting is due to Jaggi~\cite{Jaggi_FW}, establishing that with the ``standard'' stepsize sequence $h_k = \frac{2}{k+2}$, the objective gap converges at rate\footnote{Note that while the rate proven by Jaggi describes the objective gap at the terminal iterate, it holds as an equally valid bound for our objective $\min_k f(x_k) - f(x_\star)$.} at least $\frac{2LD^2}{N+2}$. In \cite[Section 4.5]{Interpolation2}, Taylor et al. previously showed through PEP experiments that this rate is slack.  The best known lower bound for Frank-Wolfe over convex sets is $\frac{LD^2}{4N}$, from Lan~\cite{Lan_LinearOptOracle}, establishing potential to improve by up to a factor of eight.

Given more structure on $C$ and/or $f$, one could hope for further gains.} Frank-Wolfe's worst-case convergence rate is known to accelerate up to $O(1/N^2)$ in a few settings: given a strongly convex $f$ and $C$~\cite{garber2015faster,Wirth_FWAccel} or given a strongly convex $f$ and $x_\star\in\mathrm{int}_{\delta} C$~\cite{Guelat1986,Wolfe,Wirth_FWAccel}. However, it remains undetermined if acceleration can be achieved in any intermediate settings, including any case with smooth constraint sets.

\begin{figure}
    \centering
    \includegraphics[width=1.0\textwidth]{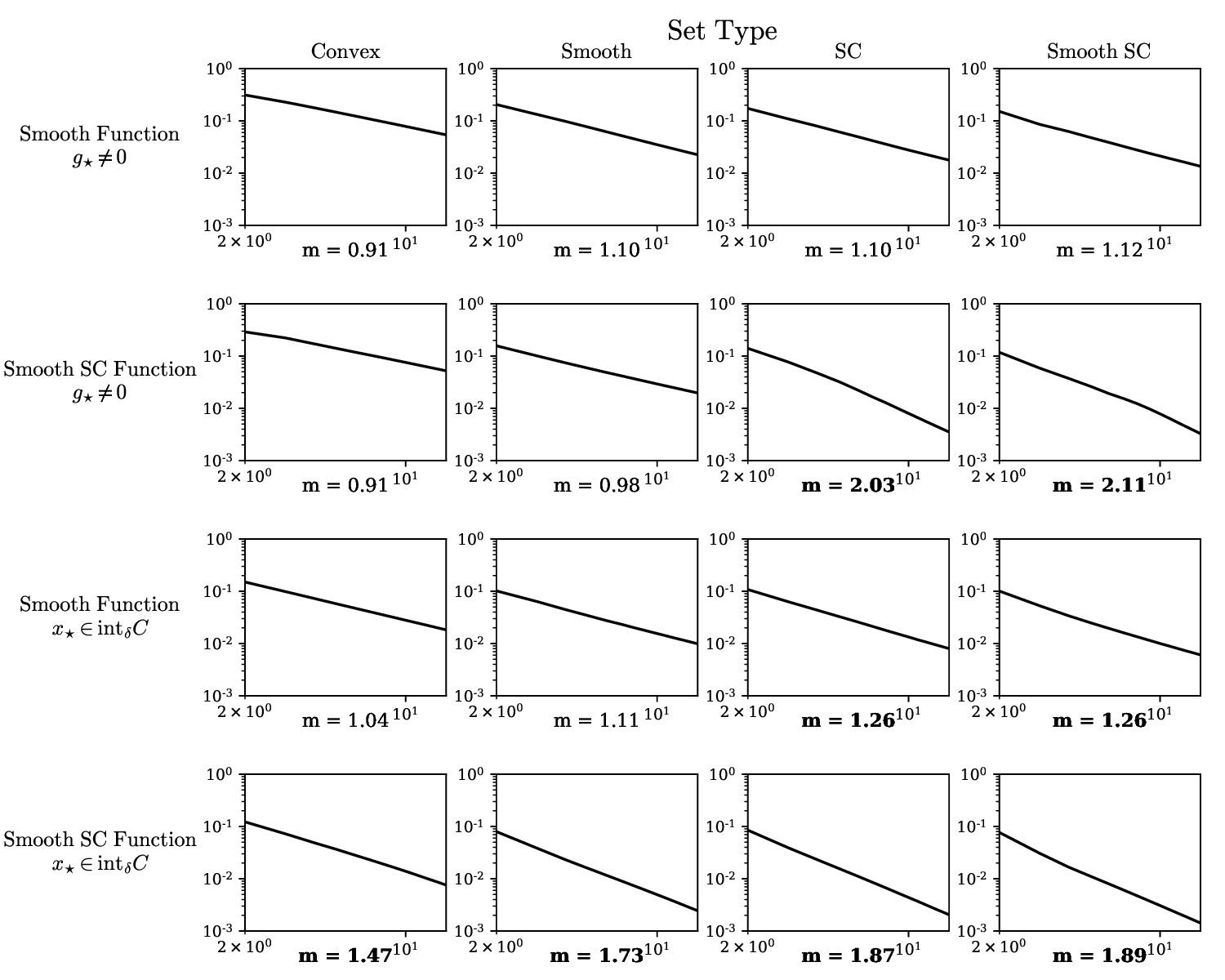}
    \caption{Survey of Frank-Wolfe's performance with standard stepsizes on $16$ different problem \edit{classes. Experiments are run with stepsizes $h_k = \frac{2}{k+2}$ and with $L=1$, $\mu=0.5$, $\beta=5$, $\alpha=1$, $D = 1$ and $\delta = 0.25$ where applicable. We} highlight in bold the problem settings in which the estimated rate $O(1/N^m)$ noticeably exceeds the standard $O(1/N)$ rate.}
    \label{Fig:FWAllSettings}
\end{figure}

Using our expanded performance estimation toolbox, we explore \edit{$16$} permutations of problem instances: $C$ being smooth or not; $C$ being strongly convex or not; $f$ being strongly convex or not; the minimizer having \edit{$g_\star\neq 0$} \edit{or having} $g_\star=0$ with $x_\star$ strictly interior to $C$. This enables us to assess if other settings achieve $O(1/N^2)$ convergence. \edit{Note that line-search or alternate assumptions on $f$ (e.g., uniformly bounded gradients: $\|\nabla f(x)\| \geq \edit{\epsilon > 0}$ for all $x \in C$) have shown improved (and sometimes linear) convergence~\cite{Wolfe,garber2015faster,WirthPena_AffineFW,Lan_LinearOptOracle} but lie outside the scope of this preliminary survey.} In \cref{Fig:FWAllSettings}, we show our results for each of these settings for $N=2,\dots, 15$ \edit{with stepsizes $h_k = \frac{2}{k+2}$}, and estimate a rate $O(1/N^m)$ by regression on the values from $N=8,\dots, 15$. 

Our results show no clear evidence of any new problem settings that achieve $O(1/N^2)$ acceleration, \edit{with most problem settings} having $m\approx 1$. However, in the settings with smooth functions, strongly convex sets, and an interior optimal point, our results seem to noticeably outperform the standard $O(1/N)$ rate, instead being closer to $O(1/N^{1.2})$. This presents an intriguing area for future analysis. \edit{It is worth noting that some known acceleration results (e.g.,~\cite{Wirth_FWAccel}) require a burn-in phase, so} it remains possible that acceleration in some of these interim settings could be missed by our survey.

\begin{remark}[Improved Convergence Rates from Set Smoothness] \edit{We now briefly consider Frank-Wolfe's performance over $\beta$-smooth sets. To date, no theory shows Frank-Wolfe benefits from smoothness in the constraint set $C$. Since the lower bound $\frac{LD^2}{4N}$ of~\cite{Lan_LinearOptOracle} uses the simplex as a constraint set (which is nonsmooth), smoothness has the potential to break this bound. To focus on the effect of smoothness,} we fix $\alpha=0$, enabling the necessary and sufficient \cref{Thm:Interpolation_NoSlack} to apply.
\edit{\cref{Fig:WirthComparison}(a) provides an immediate positive answer to whether smoothness of constraints improves Frank-Wolfe's performance.} Local solves for our exact, nonconvex PEP formulation show up to an order of magnitude improvement in performance by $N=15$ over the convex setting $\beta=\infty$. \edit{For $\beta=2$ smooth sets, performance strictly outperforms Lan's lower bound.}
\end{remark}

\begin{remark}[Gaps in Accelerated Theory for Strongly Convex Functions] In~\cite{garber2015faster}, Garber and Hazan proved that for strongly convex functions, Frank-Wolfe attains an accelerated convergence rate of $O(1/N^2)$ if the constraint set $C$ is strongly convex ($\alpha > 0$), or if $x_\star \in \Interior{\delta} C$ ($\delta>0$). Their result relied on stepsize selection by \edit{linesearch, but}
Wirth et al.\ later proved in \cite{Wirth_FWAccel} that similar $O(1/N^2)$ convergence rates were attained in these settings using the open-loop sequence $h_k = \frac{4}{k+4}$. Specifically, in \cite[Theorem E.1]{Wirth_FWAccel}, if $C$ is additionally $\alpha$-strongly convex, then Frank-Wolfe has
\begin{equation}\label{Eqn:WirthRate_SCSet}
    f(x_N) - f(x_\star) \leq \frac{\frac{128L^2}{\alpha^2\mu} + 8LD^2}{(N+2)^2}  \ .
\end{equation}
Alternatively, in \cite[Theorem 3.6]{Wirth_FWAccel}, they show that if $x_\star \in \Interior{\delta} C$, then letting $M = \lceil 64LD^2/(\mu \delta^2) \rceil$, for any $N\geq M$, Frank-Wolfe satisfies
\begin{equation}\label{Eqn:WirthRate_Interior}
    f(x_N) - f(x_\star) \leq \max \left\{ \frac{(M+3)^2}{(N+2)^2}(f(x_M) - f(x_\star)), \quad \frac{\frac{128L^2D^6}{\mu \delta^4} + 8LD^2}{(N+2)^2} \right\} \ .
\end{equation}

In \cref{Fig:WirthComparison}(b), we compare the PEP result for $h_k = \frac{4}{k+4}$ ($\ell = 4$) with the convergence rate \eqref{Eqn:WirthRate_SCSet}. \edit{In this setting of strongly convex sets with bounded diameter, we can no longer apply our necessary and sufficient interpolation of \cref{Thm:Interpolation_NoSlack}. As a consequence, by applying our interpolation conditions of $\mathrm{Interp}(\alpha, 0, D; 1)$, we only obtain upper bounds on the method's worst-case performance.} These approximate PEP results outperform the guarantee of \eqref{Eqn:WirthRate_SCSet} by about two orders of magnitude.
Additionally, we include in \cref{Fig:WirthComparison}(b) the performance for other values of $\ell$ (with $h_k = \frac{\ell}{k + \ell}$), seeing $\ell = 2$ performed best among the tested values and, in particular, outperforms $\ell = 4$.
\edit{We repeat} this analysis for the setting of a strongly convex function with an interior optimal point instead of a strongly convex constraint set. Note our necessary and sufficient \cref{Thm:Interpolation_NoSlack} then applies. Our results in \cref{Fig:WirthComparison}(c) indicate that much tighter analysis is possible here as well\footnote{Since \eqref{Eqn:WirthRate_Interior} only applies after $M = \lceil 64LD^2/(\mu\delta^2)\rceil \geq 256$ iterations, we cannot compare directly with \eqref{Eqn:WirthRate_Interior} due to our computational limits around $N=15$.} with $\ell = 1$ performing best.
\end{remark}

\begin{figure}[t]
    \centering
    \includegraphics[width=1.0\textwidth]{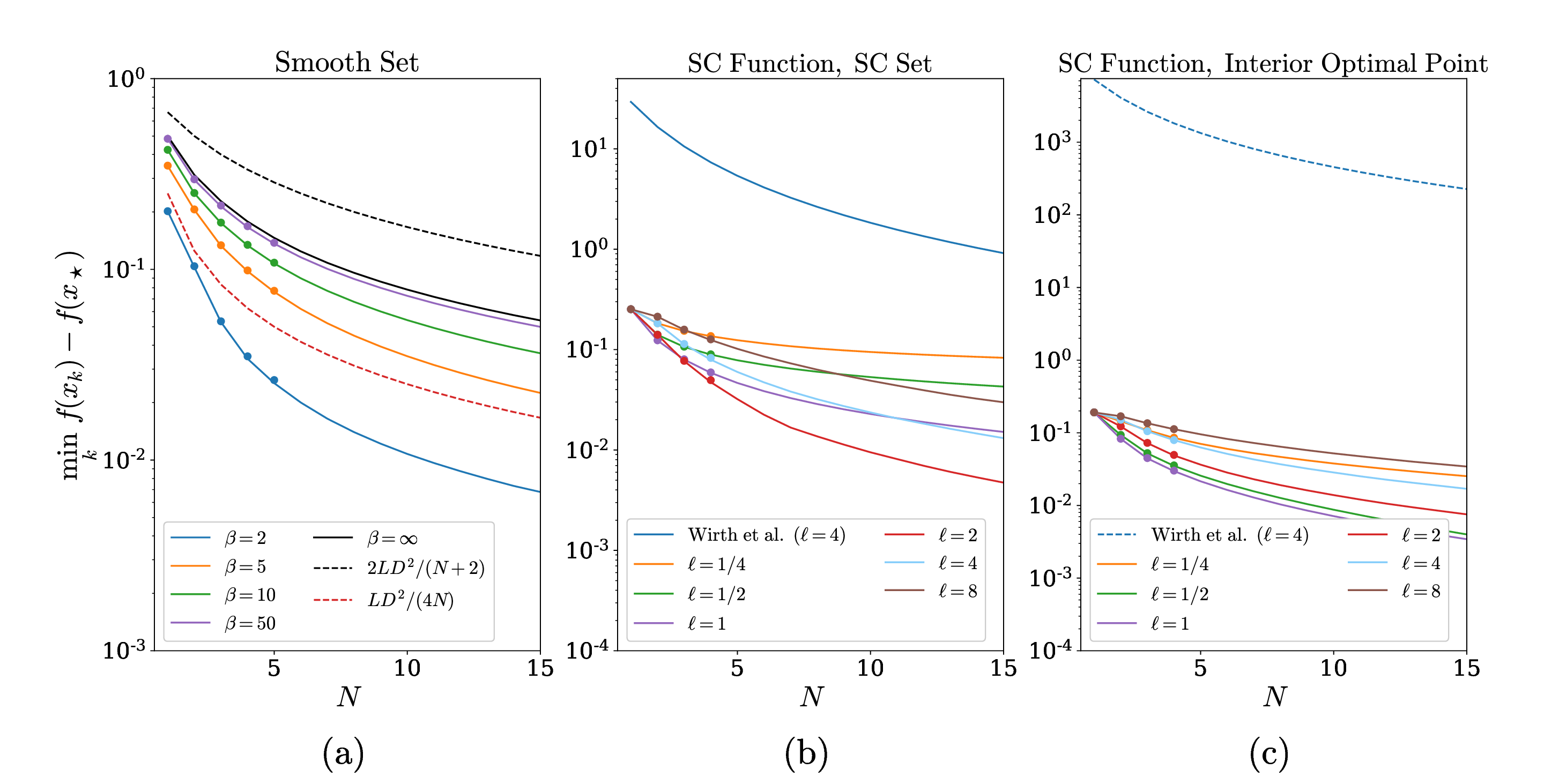}
    \caption{\edit{(a) PEP results for Frank-Wolfe over $\beta$-smooth sets with  $h_k = \frac{2}{k+2}$ and for comparison, the lower and upper bounds of \cite{Jaggi_FW} and \cite{Lan_LinearOptOracle} for convex sets. (b) Results for $\mu$-strongly convex functions over $\alpha$-strongly convex sets with $h_k = \frac{\ell}{k+\ell}$ and for comparison, the bound \eqref{Eqn:WirthRate_SCSet}, fixing $\mu = 0.5$ and $\alpha = 1.0$. (c) PEP results for $\mu$-strongly convex functions with $x_\star \in \Interior{\delta} C$ with $h_k = \frac{\ell}{k+\ell}$, fixing $\mu=0.5$ and $\delta = 0.25$. Although the bound~\eqref{Eqn:WirthRate_Interior} does not apply until $N \geq 256$, it is included, dotted. All results above are obtained via local optimization but verified globally for small $N$, as indicated by individual points on the plot.}}
    \label{Fig:WirthComparison}
\end{figure}

    \section{Gradient Methods for Epismooth Functions}\label{Sec:EpiSmooth}
As a final application, we show that our set interpolation theory can provide novel insights even in the context of unconstrained minimization. In particular, consider a gradient method \edit{\eqref{Eqn:GM} applied to minimize some differentiable, convex function $f:\R^d \to \R\cup\{\infty\}$.}
Much of the modern theory for gradient methods assumes $f$ is $L$-smooth. This is a major restriction from the family of all differentiable functions, ruling out, for instance, all polynomials with degree greater than two.

Here, we consider applying first-order methods to a larger class of differentiable functions, namely functions with a smooth epigraph. 
We say $f$ is {\bf $L$-epigraphically smooth} ({\bf $L$-epismooth}) if its epigraph $\mathrm{epi\ }f = \{(x,t) \mid f(x)\leq t\}$ is an $L$-smooth set. \cref{Thm:SetFuncInterp_Smooth} with $\eta=1$ establishes that all $L$-smooth functions are $L$-epismooth, so this is a strictly broader class of functions.
Note epismoothness of $f$ does not guarantee that $f$ has a full domain of $\R^d$ (\edit{consider a simple U-shaped function}). Consequently, one cannot guarantee the iteration~\eqref{Eqn:GM} is well-defined in general, so additional care is needed. To handle this, we consider modified gradient methods of the form
\begin{equation} \label{Eqn:Epi-GM}
    x_{i+1} = x_i - \frac{1}{L}\sum_{j=0}^i H_{i,j} \frac{\nabla f(x_j)}{\sqp{\|\nabla f(x_j)\|^2}} \ .    \tag{Epi-GM}
\end{equation}
Normalizing our stepsizes above allows methods to keep $x_{i+1}-x_i$ bounded even as $\nabla f(x_k)$ becomes arbitrarily large (as shown in \cref{Fig:EpiSmoothGD_Diagram}). This property can then be used to ensure the iterates do not leave the domain of $f$: For example, fixing $H$ to be the identity matrix, the above iteration becomes the gradient descent iteration $x_{i+1} = x_i - \frac{\nabla f(x_i)}{L \sqp{\|\nabla f(x_i)\|^2}}$, which repeatedly moves to the minimizer of the ball majorizing $f$, ensuring progress every iteration.

\begin{figure}
    \centering
    \scalebox{0.70}{\tikzset{every picture/.style={line width=0.75pt}} 

\begin{tikzpicture}[x=0.75pt,y=0.75pt,yscale=-1,xscale=1]


\draw [color={rgb, 255:red, 175; green, 175; blue, 175 }  ,draw opacity=1 ][line width=1.5]    (74.23,40.73) .. controls (78.11,103.62) and (82.73,358.12) .. (265.99,357.94) .. controls (451.12,357.13) and (445.13,109.16) .. (448.88,41.13) ;

\draw [shift={(448.99,39.12)}, rotate = 93.52] [color={rgb, 255:red, 175; green, 175; blue, 175 }  ,draw opacity=1 ][line width=1.5]    (14.21,-4.28) .. controls (9.04,-1.82) and (4.3,-0.39) .. (0,0) .. controls (4.3,0.39) and (9.04,1.82) .. (14.21,4.28)   ;
\draw [shift={(73.99,37.12)}, rotate = 85.91] [color={rgb, 255:red, 175; green, 175; blue, 175 }  ,draw opacity=1 ][line width=1.5]    (14.21,-4.28) .. controls (9.04,-1.82) and (4.3,-0.39) .. (0,0) .. controls (4.3,0.39) and (9.04,1.82) .. (14.21,4.28)   ;

\draw  [dash pattern={on 4.5pt off 4.5pt}]  (160,160.28) -- (160,311) ;
\draw [shift={(160,313)}, rotate = 270] [color={rgb, 255:red, 0; green, 0; blue, 0 }  ][line width=0.75]    (10.93,-3.29) .. controls (6.95,-1.4) and (3.31,-0.3) .. (0,0) .. controls (3.31,0.3) and (6.95,1.4) .. (10.93,3.29)   ;
\draw    (92,180.67) -- (158.08,160.85) ;
\draw [shift={(160,160.28)}, rotate = 163.31] [color={rgb, 255:red, 0; green, 0; blue, 0 }  ][line width=0.75]    (10.93,-3.29) .. controls (6.95,-1.4) and (3.31,-0.3) .. (0,0) .. controls (3.31,0.3) and (6.95,1.4) .. (10.93,3.29)   ;
\draw  [fill={rgb, 255:red, 0; green, 0; blue, 0 }  ,fill opacity=1 ] (88,180.67) .. controls (88,178.46) and (89.79,176.67) .. (92,176.67) .. controls (94.21,176.67) and (96,178.46) .. (96,180.67) .. controls (96,182.88) and (94.21,184.67) .. (92,184.67) .. controls (89.79,184.67) and (88,182.88) .. (88,180.67) -- cycle ;
\draw   (90,168.67) .. controls (90,144.92) and (109.25,125.67) .. (133,125.67) .. controls (156.75,125.67) and (176,144.92) .. (176,168.67) .. controls (176,192.41) and (156.75,211.67) .. (133,211.67) .. controls (109.25,211.67) and (90,192.41) .. (90,168.67) -- cycle ;
\draw  [fill={rgb, 255:red, 0; green, 0; blue, 0 }  ,fill opacity=1 ] (129,168.67) .. controls (129,166.46) and (130.79,164.67) .. (133,164.67) .. controls (135.21,164.67) and (137,166.46) .. (137,168.67) .. controls (137,170.88) and (135.21,172.67) .. (133,172.67) .. controls (130.79,172.67) and (129,170.88) .. (129,168.67) -- cycle ;
\draw    (92,180.67) -- (51.9,194.03) ;
\draw [shift={(50,194.67)}, rotate = 341.57] [color={rgb, 255:red, 0; green, 0; blue, 0 }  ][line width=0.75]    (10.93,-3.29) .. controls (6.95,-1.4) and (3.31,-0.3) .. (0,0) .. controls (3.31,0.3) and (6.95,1.4) .. (10.93,3.29)   ;
\draw [color={rgb, 255:red, 155; green, 155; blue, 155 }  ,draw opacity=1 ]   (39,138) .. controls (27.24,161.19) and (64.46,159.74) .. (86.66,179.43) ;
\draw [shift={(88,180.67)}, rotate = 223.67] [color={rgb, 255:red, 155; green, 155; blue, 155 }  ,draw opacity=1 ][line width=0.75]    (10.93,-3.29) .. controls (6.95,-1.4) and (3.31,-0.3) .. (0,0) .. controls (3.31,0.3) and (6.95,1.4) .. (10.93,3.29)   ;
\draw [color={rgb, 255:red, 155; green, 155; blue, 155 }  ,draw opacity=1 ]   (66,242.67) .. controls (61.15,229.09) and (65.71,219.27) .. (74.2,200.44) ;
\draw [shift={(75,198.67)}, rotate = 114.23] [color={rgb, 255:red, 155; green, 155; blue, 155 }  ,draw opacity=1 ][line width=0.75]    (10.93,-3.29) .. controls (6.95,-1.4) and (3.31,-0.3) .. (0,0) .. controls (3.31,0.3) and (6.95,1.4) .. (10.93,3.29)   ;
\draw  [fill={rgb, 255:red, 0; green, 0; blue, 0 }  ,fill opacity=1 ] (156,160.28) .. controls (156,158.07) and (157.79,156.28) .. (160,156.28) .. controls (162.21,156.28) and (164,158.07) .. (164,160.28) .. controls (164,162.49) and (162.21,164.28) .. (160,164.28) .. controls (157.79,164.28) and (156,162.49) .. (156,160.28) -- cycle ;
\draw [color={rgb, 255:red, 155; green, 155; blue, 155 }  ,draw opacity=1 ]   (218,109.28) .. controls (210.24,141.29) and (190.24,147.75) .. (170.8,155.41) ;
\draw [shift={(168.99,156.12)}, rotate = 338.2] [color={rgb, 255:red, 155; green, 155; blue, 155 }  ,draw opacity=1 ][line width=0.75]    (10.93,-3.29) .. controls (6.95,-1.4) and (3.31,-0.3) .. (0,0) .. controls (3.31,0.3) and (6.95,1.4) .. (10.93,3.29)   ;
\draw  [fill={rgb, 255:red, 0; green, 0; blue, 0 }  ,fill opacity=1 ] (156,318) .. controls (156,315.79) and (157.79,314) .. (160,314) .. controls (162.21,314) and (164,315.79) .. (164,318) .. controls (164,320.21) and (162.21,322) .. (160,322) .. controls (157.79,322) and (156,320.21) .. (156,318) -- cycle ;
\draw [line width=1.5]    (34,358) -- (517,358) ;
\draw [shift={(521,358)}, rotate = 180] [fill={rgb, 255:red, 0; green, 0; blue, 0 }  ][line width=0.08]  [draw opacity=0] (11.61,-5.58) -- (0,0) -- (11.61,5.58) -- cycle    ;
\draw [shift={(30,358)}, rotate = 0] [fill={rgb, 255:red, 0; green, 0; blue, 0 }  ][line width=0.08]  [draw opacity=0] (11.61,-5.58) -- (0,0) -- (11.61,5.58) -- cycle    ;
\draw [line width=1.5]    (265,44) -- (264.99,387.12) ;
\draw [shift={(265,40)}, rotate = 90] [fill={rgb, 255:red, 0; green, 0; blue, 0 }  ][line width=0.08]  [draw opacity=0] (11.61,-5.58) -- (0,0) -- (11.61,5.58) -- cycle    ;

\draw (9,115.4) node [anchor=north west][inner sep=0.75pt]  [font=\large]  {$( x_{k} ,f_{k})$};
\draw (9,250.4) node [anchor=north west][inner sep=0.75pt]  [font=\large]  {$\frac{( \nabla f( x_{k}) ,-1)}{L \ \| ( \nabla f( x_{k}) ,-1) \| }$};

\draw (\branchEpiCoord,64.4) node [anchor=north west][inner sep=0.75pt]  [font=\large]  {$\left( x_{k} -\frac{h_{k} \nabla f( x_{k})}{L \sqp{\|  \nabla f( x_{k}) \| ^{2}}} ,\ f_{k} +\frac{h_{k}}{L \sqp{\| \nabla f( x_{k}) \| ^{2}}}\right)$};
\draw (170,306.4) node [anchor=north west][inner sep=0.75pt]  [font=\large]  {$( x_{k+1} ,\ f_{k+1})$};

\end{tikzpicture}}
    \caption{Modified gradient method for epismooth functions \eqref{Eqn:GM}.}
    \label{Fig:EpiSmoothGD_Diagram}
\end{figure}

To illustrate the breadth of this class of functions, the following proposition provides a useful means of checking if a function is \textbf{locally} epismooth.
\begin{proposition} \label{Prop:EpiSmooth-Properties}
    For any $f \in \mathcal{C}^2$ and $x$ in the domain of $f$, let
    \begin{equation} \label{Eqn:epismooth-hessian}
        M(x) = \frac{(I + \nabla f(x) \nabla f(x)^T)^{-1/2} (\nabla^2 f(x)) (I + \nabla f(x) \nabla f(x)^T)^{-1/2}}{\sqp{\|\nabla f(x)\|^2}} \ .
    \end{equation}
    Then $f$ is locally $L$-epismooth at $x$ if and only if $L \geq \lambda_\mathrm{max} (M)$.
\end{proposition}

\begin{proof}
    \edit{
We define a local ball upper bound of $f$ at $x$ by
\begin{equation}
    f(y) \leq \ballUB{x}(y; L) := f(x) + \frac{1}{L\sqp{\|\nabla f(x)\|^2}} - \sqrt{\frac{1}{L^2} - \|y - x + \frac{\nabla f(x)}{L \sqp{\|\nabla f(x)\|^2}} \|^2}\label{Eqn:ballUpperBound}
\end{equation}
where we say $\ballUB{x}(y; L) = \infty$ outside of its domain.
}
From the definition of smoothness, $f$ is $L$-epismooth if and only if for all $x,y$, $f(y) \leq \ballUB{x}(y; L)$.
    Observe that $\ballUB{x}(x; L) = f(x)$ and $\nabla \ballUB{x}(x; L) = \nabla f(x)$. Therefore, the upper bound $\ballUB{x}(y; L)$ holds locally if and only if $\nabla^2 f(x) \preceq \nabla^2 \ballUB{x}(x; L)$. We can compute $\nabla^2 \ballUB{x}(x; L) = L \sqp{ \|\nabla f(x)\|^2}(I + \nabla f(x) \nabla f(x)^T)$. Then rearranging and taking the maximum eigenvalue, we obtain our result. \myqed
\end{proof}

Using this result, we can identify simple functions that are not smooth in the function sense of uniformly Lipschitz gradient but satisfy local epismoothness on their domain. For example, consider
\begin{equation*}
    f(x) = \begin{cases}
        - \log x \quad & \mathrm{if\ } x>0 \\
        \infty & \mathrm{if\ } x \leq 0 \ .
    \end{cases}
\end{equation*}
Using \eqref{Eqn:epismooth-hessian}, we can calculate $\lambda_\mathrm{max} (M(x)) = \frac{\frac{1}{x^2}}{(1 + \frac{1}{x^2})^{3/2}} \leq \frac{2}{3\sqrt{3}}$ for all $x>0$.
So $f$ is locally $\frac{2}{3\sqrt{3}}$-epismooth everywhere on its domain. \edit{As a second, nontrivial example, one can verify for example $\|Ax-b\|^4_4$ is epismooth.}

\subsection{Worst-Case Performance of Epismooth Gradient Methods}
The worst-case performance of a gradient method defined by a predetermined stepsize matrix $H$ on a smooth function and on an epismooth function are closely related. \edit{In \cref{Sec:Preliminaries}, we introduced the PEP for smooth functions \eqref{Eqn:basicGD-example1}, and we now present an analogous PEP for the epismooth case.}
In both cases, we assume the initial point $x_0$ satisfies $\|x_0 - x_\star\| \leq R$ for some minimizer $x_\star$ of $f$. Lastly, to ensure that $x_0$ is in the domain of $f$ (i.e., $f(x_0) < \infty$), we further require that $R \leq \ol$. Then the worst-case performance of~\eqref{Eqn:Epi-GM} for a given matrix $H$ on an $L$-epismooth function is
\begin{equation}\label{Eqn:OptEpiGDInfinite}
    \pEpiSmooth(L,R) = \begin{cases} \max\limits_{x_i,f} \quad & f(x_N) - f(\xStar) \\
        \text{s.t.} & \epi f \in \C_{0,L,\infty} \\
        & x_{i+1} = x_i - \frac{1}{L}\sum_{j=0}^{i} H_{i,j} \frac{\nabla f (x_j)}{\sqp{\|\nabla f(x_j)\|^2}} \quad \quad \forall i \in [0\mathrm{\,:\,}N-1] \\
        & \|x_0 - \xStar \| \leq R, \qquad \nabla f(\xStar) = 0 \ .
    \end{cases}
\end{equation}
For simplicity, we have suppressed the parameters $N$ and $H$ from our notation for $\pEpiSmooth$.
\edit{We again denote the PEP for performance of~\eqref{Eqn:GM} given $H$ on an $L$-smooth function by $p_{\mathrm{S}}(L, R)$}.
Both of these settings possess useful rescaling properties, proven in \cref{App:ES-PEP-rescale-proof}.
\begin{lemma} \label{Lem:PEP-rescaling}
    For all $\eta>0$ and any choice of $N,L,R$, one has that
    \begin{align*}
        \begin{array}{ll} \begin{array}{l} p_\mathrm{S}(L, \eta R) = \eta^2 p_\mathrm{S}(L, R) \ , \vspace{0.1cm} \\
        p_\mathrm{S}(\eta L, R) = \eta p_\mathrm{S}(L, R) \ ,
        \end{array} & \qquad\qquad  
        \pEpiSmooth(L, \eta R) = \eta \pEpiSmooth(\eta L,R)
        \ .
        \end{array}
    \end{align*}
\end{lemma}
Observe that together the rescaling properties for $p_\mathrm{S}$ imply $p_\mathrm{S}(L,\eta R) = \eta p_\mathrm{S}(\eta L, R)$, matching that of the epismooth setting. However, notably, the individual rescaling properties for $p_\mathrm{S}$ do not hold for $\pEpiSmooth$. These individual rescaling properties of $p_\mathrm{S}$ establish that it suffices to characterize $p_\mathrm{S}(1,1)$ to fully understand $p_\mathrm{S}$. Since this does not hold for $\pEpiSmooth$, any future characterizations of epismooth minimization methods must depend on the ratio of $L$ and $R$.

Under an appropriate regularity condition, we find that the convergence of any epismooth minimization method converges to its classic smooth convergence rate as the ratio between $L$ and $R$ grows. That is, if initialized sufficiently close to $x_\star$, perhaps from some initial ``burn-in'' procedure, guarantees proven for $L$-smooth minimization can be lifted to $L$-epismooth minimization.

Formally, this result relies on the following regularity condition:  A gradient method defined by $H$ is \textbf{eventually-epismooth-stable} if for any $L$, there exist constants $\bar{R},C > 0$ such that applying~\eqref{Eqn:Epi-GM} to any $L$-epismooth function with $\|x_0 - x_\star\| \leq R < \bar{R}$ must have $\|x_i-x_\star\| \leq C R$, $\|g_i\| \leq C L R$, and $|f_i| \leq C LR^2$ for all $i=0,\dots,N$.
This condition ensures the following asymptotic relation, proven in \cref{App:ES-Proof-Epi-PEP-Convergence}.
\begin{theorem} \label{Thm:Epi-Smooth-PEP-Convergence}
    For any eventually-epismooth-stable method defined by $H$,
    \begin{equation*}
        \lim_{\eta \to 0} \frac{\pEpiSmooth(L, \eta R)}{\eta^2} = \lim_{\eta \to 0} \frac{p_\mathrm{S}(L, \eta R)}{\eta^2} = p_\mathrm{S}(L, R) \ .
    \end{equation*}
\end{theorem}
Hence, after an initial burn-in producing $x_0$ with $\|x_0 - x_\star\|$ sufficiently small, optimization strategies for smooth functions such as momentum \cite{Nesterov,OGM} or recent long-step gradient descent techniques \cite{Silver_Accel,Grimmer_Accel} may be applicable to epismooth functions. One simple example of an eventually-epismooth-stable algorithm is \edit{applying \eqref{Eqn:Epi-GM} with stepsizes $h \in (0,2)$ (i.e., $H = \mathrm{diag} (h,\dots, h)$)}. One can directly show this iteration guarantees descent and a nonincreasing distance to $x_\star$, as illustrated in \cref{Fig:EpiSmoothGD_Diagram}, ensuring the needed bounds. 

\edit{Applying our interpolation theory, we verify this convergence of epismooth and smooth worst-case performances}.
Following \edit{\cref{Thm:Interpolation_NoSlack},} the same Grammian reformulation presented in previous \edit{sections provides} a reformulation \edit{of $\pEpiSmooth(L,R)$ as an SDP with an additional rank-1 constraint. We show a brief derivation in \cref{App:EpiSDPDerivation}.
Unfortunately, the rank-1 constraint makes this problem nonconvex.} Hence we approach it directly as a QCQP using black-box optimization software~\cite{JuMP,gurobi}.

\begin{figure}[t]
    \centering
    \includegraphics[width=1.0\textwidth]{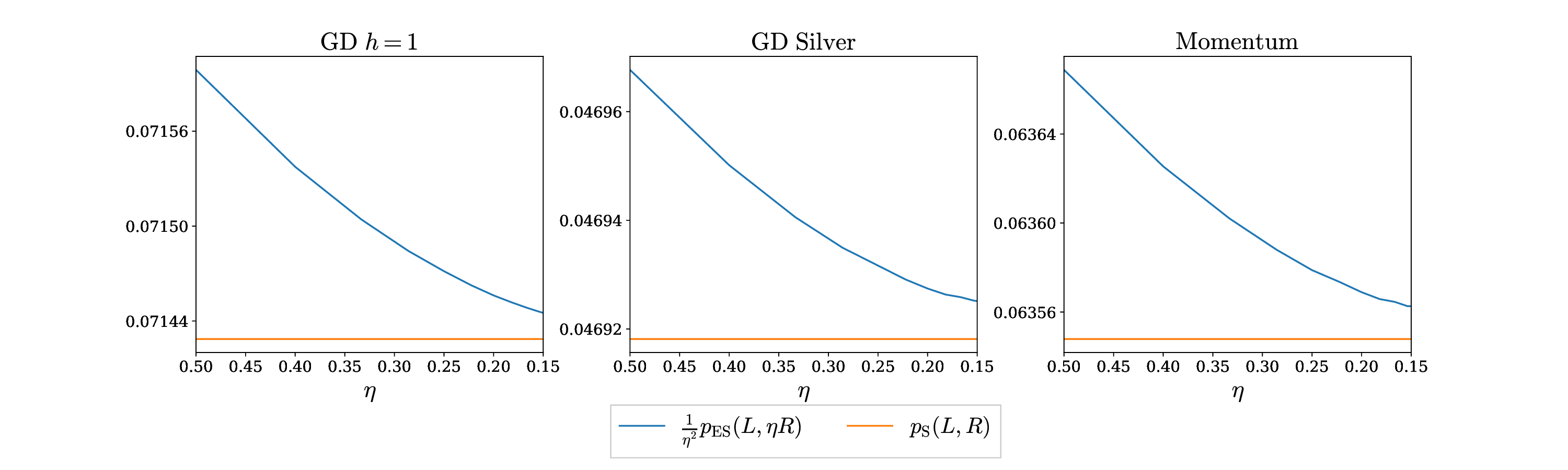}
    \caption{\edit{Convergence} of $\frac{1}{\eta^2}\pEpiSmooth(L,\eta R)$ and $p_\mathrm{S}(L,R)$ as $\eta$ decreases, fixing $N=3$. Data is shown for three different~\eqref{Eqn:Epi-GM} algorithms: gradient descent with $h=1$, gradient descent with silver stepsizes~\cite{Silver_Accel}, and Nesterov fast gradient method~\cite{Nesterov}.}
    \label{Fig:EpiSmoothResults}
\end{figure}

\edit{\cref{Fig:EpiSmoothResults}} considers the worst-case performance of several classic smooth optimization methods. For the stepsize matrices $H$ corresponding to gradient descent (with $h_i=1$), gradient descent with Silver stepsizes, and Nesterov's fast gradient method, \edit{the worst-case behavior of \eqref{Eqn:GM} and \eqref{Eqn:Epi-GM} over their respective function classes converge as $\eta\rightarrow 0$}, agreeing with \cref{Thm:SetFuncInterp_Smooth}.

    \paragraph{Acknowledgments.} This work was supported in part by the Air Force Office of Scientific Research under award number FA9550-23-1-0531. Benjamin Grimmer was additionally supported as a fellow of the Alfred P. Sloan Foundation.

    {\small
    \bibliographystyle{unsrt}
    \bibliography{references}

@article{Lessard2016,
author = {Lessard, Laurent and Recht, Benjamin and Packard, Andrew},
title = {Analysis and Design of Optimization Algorithms via Integral Quadratic Constraints},
journal = {SIAM Journal on Optimization},
volume = {26},
number = {1},
pages = {57-95},
year = {2016},
doi = {10.1137/15M1009597},

URL = { 
    
        https://doi.org/10.1137/15M1009597
    
    

},
eprint = { 
    
        https://doi.org/10.1137/15M1009597
    
    

}
}

@inproceedings{garber2015faster,
  title = 	 {Faster Rates for the {F}rank-{W}olfe Method over Strongly-Convex Sets},
  author = 	 {Garber, Dan and Hazan, Elad},
  booktitle = 	 {Proceedings of the 32nd International Conference on Machine Learning},
  pages = 	 {541--549},
  year = 	 {2015},
  volume = 	 {37},
  publisher =    {PMLR}
}

@article{LiuGrimmer,
      title={Gauges and Accelerated Optimization over Smooth and/or Strongly Convex Sets}, 
      author={Ning Liu and Benjamin Grimmer},
      year={2023},
      eprint={2303.05037},
      archivePrefix={arXiv},
    journal={arXiv:2303.05037}
}

@article{FirstPEP,
author = {Drori, Yoel and Teboulle, Marc},
title = {Performance of first-order methods for smooth convex minimization: A novel approach},
year = {2014},
issue_date = {June 2014},
volume = {145},
number = {1–2},
pages = {451-482},
journal = {Mathematical Programming},
}

@article{Interpolation,
author = {Taylor, Adrien and Hendrickx, Julien and Glineur, François},
year = {2017},
pages = {307-345},
title = {Smooth strongly convex interpolation and exact worst-case performance of first-order methods},
volume = {161},
journal = {Mathematical Programming}
}

@article{Interpolation2,
author = {Taylor, Adrien and Hendrickx, Julien and Glineur, Fran\c{c}ois},
title = {Exact Worst-Case Performance of First-Order Methods for Composite Convex Optimization},
journal = {SIAM Journal on Optimization},
volume = {27},
number = {3},
pages = {1283-1313},
year = {2017}
}

@article{ShuvoBnB,
    title={Branch-and-bound performance estimation
    programming: A unified methodology for constructing optimal optimization method},
    year = {2024},
    pages = {567-639},
    volume = {204},
    journal = {Mathematical Programming},
    author = {Das Gupta, Shuvomoy and {Van Parys}, Bart P.G. and Ryu, Ernest K.}
}

@article{Grimmer_Accel,
      title={Accelerated Objective Gap and Gradient Norm Convergence for Gradient Descent via Long Steps}, 
      author={Benjamin Grimmer and Kevin Shu and Alex L. Wang},
      year={2025},
    journal={INFORMS Journal on Optimization},
    volume = {7},
    issue = {2},
    pages = {156-169}
}

@article{Grimmer_Accel_Original,
      title={Accelerated Gradient Descent via Long Steps}, 
      author={Benjamin Grimmer and Kevin Shu and Alex L. Wang},
      year={2023},
      eprint={2309.09961},
      archivePrefix={arXiv},
      primaryClass={math.OC},
    journal = {arXiv:2309.09961}
}

@article{Silver_Accel,
      title={Acceleration by Stepsize Hedging: Silver Stepsize Schedule for Smooth Convex Optimization}, 
      author={Jason M. Altschuler and Pablo A. Parrilo},
    volume={213},
    pages = {1105-1118},
    year = {2025},
    journal = {Mathematical Programming}
}

@article{Silver_Accel_General,
author = {Altschuler, Jason M. and Parrilo, Pablo A.},
title = {Acceleration by Stepsize Hedging: Multi-Step Descent and the Silver Stepsize Schedule},
year = {2025},
issue_date = {April 2025},
publisher = {Association for Computing Machinery},
address = {New York, NY, USA},
volume = {72},
number = {2},
issn = {0004-5411},
journal = {J. ACM},
month = mar,
articleno = {12},
numpages = {38},
}

@article{Drori_OGMOptimal,
title = {The exact information-based complexity of smooth convex minimization},
journal = {Journal of Complexity},
author = {Yoel Drori},
volume = {39},
pages = {1-16},
year = {2017},
}

@article{OGM,
   title={Optimized first-order methods for smooth convex minimization},
   volume={159},
   ISSN={1436-4646},
   number={1–2},
   journal={Mathematical Programming},
   author={Kim, Donghwan and Fessler, Jeffrey A.},
   year={2016},
   pages={81–107}
}

@article{pepit2022,
  title={{PEPit}: computer-assisted worst-case analyses of first-order optimization methods in {P}ython},
  author={Goujaud, Baptiste and Moucer, C\'eline and Glineur, Fran\c{c}ois and Hendrickx, Julien and Taylor, Adrien and Dieuleveut, Aymeric},
  year = {2024},
    month = {08},
    pages = {1-31},
    volume = {16},
    journal = {Mathematical Programming Computation}
}

@inproceedings{Wirth_FWAccel,
      title={Acceleration of {F}rank-{W}olfe Algorithms with Open-Loop Step-Sizes}, 
      author={Elias Wirth and Thomas Kerdreux and Sebastian Pokutta},
      year={2023}, 
    booktitle = 	 {Proceedings of the 26th International Conference on Artificial Intelligence and Statistics},
    pages = 	 {77--100},
    volume = {206},
  publisher =    {PMLR},
}

@article{Vial_StrongAndWeak,
     ISSN = {0364765X, 15265471},
     author = {Jean-Philippe Vial},
     journal = {Mathematics of Operations Research},
     number = {2},
     pages = {231--259},
     publisher = {INFORMS},
     title = {Strong and Weak Convexity of Sets and Functions},
     volume = {8},
     year = {1983}
}

@article{Drori2019,
   title={Efficient first-order methods for convex minimization: a constructive approach},
   volume={184},
   number={1–2},
   journal={Mathematical Programming},
   author={Drori, Yoel and Taylor, Adrien B.},
   year={2020},
   pages={183–220}
}

@inproceedings{Jaggi_FW,
  title = 	 {Revisiting {Frank-Wolfe}: Projection-Free Sparse Convex Optimization},
  author = 	 {Jaggi, Martin},
  booktitle = 	 {Proceedings of the 30th International Conference on Machine Learning},
  pages = 	 {427--435},
  year = 	 {2013},
  publisher =    {PMLR},
}

@article{Lan_LinearOptOracle,
      title={The Complexity of Large-scale Convex Programming under a Linear Optimization Oracle}, 
      author={Guanghui Lan},
      year={2014},
      eprint={1309.5550},
      archivePrefix={arXiv},
      primaryClass={math.OC},
    journal={arXiv:1309.5550}
}

@inproceedings{Taylor_InterpolationStochastic,
      title={Stochastic first-order methods: non-asymptotic and computer-aided analyses via potential functions}, 
      author={Adrien Taylor and Francis Bach},
      year={2021},
      booktitle = {Proceedings of the Thirty-Second Conference on Learning Theory},
      pages = {2934-2992},
      publisher = {PMLR},
        volume = {99}
}

@Inbook{Goncharov2017,
    author="Goncharov, Vladimir V.
    and Ivanov, Grigorii E.",
    title="Strong and Weak Convexity of Closed Sets in a Hilbert Space",
    bookTitle="Operations Research, Engineering, and Cyber Security: Trends in Applied Mathematics and Technology",
    year="2017",
    publisher="Springer International Publishing",
    pages="259--297",
}

@article{Vial_StrongConvexity,
    title = {Strong convexity of sets and functions},
    journal = {Journal of Mathematical Economics},
    volume = {9},
    number = {1},
    pages = {187-205},
    year = {1982},
    issn = {0304-4068},
    author = {Jean-Philippe Vial},
}

@article{Jung1901,
    author = {Jung, Heinrich},
    journal = {Journal für die reine und angewandte Mathematik},
    pages = {241-257},
    title = {Ueber die kleinste Kugel, die eine räumliche Figur einschliesst},
    volume = {123},
    year = {1901},
}

@article{WirthPena_AffineFW,
      title={Accelerated Affine-Invariant Convergence Rates of the {F}rank-{W}olfe Algorithm with Open-Loop Step-Sizes}, 
      author={Elias Wirth and Javier Pena and Sebastian Pokutta},
      year={2025},
        journal = {Mathematical Programming}
}

@article{Nesterov,
  title={A method for solving the convex programming problem with convergence rate ${O}(1/k^2)$},
  author={Yurii Nesterov},
  journal={Proceedings of the USSR Academy of Sciences},
  year={1983},
  volume={269},
  pages={543-547}
}

@article{Wolfe,
title = {Convergence theory in nonlinear programming},
author = {Philip Wolfe},
journal = {Integer and Nonlinear Programming},
pages = {1-36},
year = {1970}}

@article{Guelat1986,
author = {J. GuéLat and P. Marcotte},
title = {Some comments on {W}olfe{'}s {`}away step{'}},
year = {1986},
journal = {Mathematical Programming},
volume = {35},
pages = {110-119}
}

@article{JuMP,
author = {Dunning, Iain and Huchette, Joey and Lubin, Miles},
title = {{J}u{MP}: A Modeling Language for Mathematical Optimization},
journal = {SIAM Review},
volume = {59},
number = {2},
pages = {295-320},
year = {2017}
}

@manual{mosek,
   author = "MOSEK ApS",
   title = "{MOSEK} Optimizer {API} for {J}ulia manual. Version 10.0.",
   year = 2024,
   url = "https://docs.mosek.com/latest/juliaapi/index.html"
}

@misc{gurobi,
  author = {{Gurobi Optimization, LLC}},
  title = {{Gurobi Optimizer Reference Manual}},
  year = 2024
}

@article{zamani2024,
      title={Exact convergence rate of the subgradient method by using {P}olyak step size}, 
      author={Moslem Zamani and François Glineur},
      year={2024},
      eprint={2407.15195},
      archivePrefix={arXiv},
      primaryClass={math.OC},
    journal = {arXiv:2407.15195}
}

@article{OGMG,
author = {Donghwan Kim and Jeffrey A. Fessler},
title = {Optimizing the Efficienty of First-Order Methods for Decreasing the Gradient of Smooth Convex Functions},
year = {2021},
journal = {Journal of Optimization Theory and Applications},
volume = {188},
pages = {192-219}
}

@article{ProxGradientPEP_Taylor,
    author = {Adrien B. Taylor and Julien M. Hendrickx and François Glineur},
    title = {Exact Worst-Case Convergence Rates of the Proximal Gradient Method for Composite Convex Minimization},
    journal = {Journal of Optimization Theory and Applications},
    volume = {178},
    pages = {455-476},
    year = {2018},
}

@article{ExactLinesearchPEP_Klerk,
    author = {Etienne de Klerk and François Glineur and Adrien B. Taylor},
    title = {On the worst-case complexity of the gradient method with exact line search for smooth strongly convex functions},
    journal = {Optimization Letters},
    volume = {11},
    pages = {1185-1199},
    year = {2017},
}

@article{ITEM_Taylor,
    author = {Adrien Taylor and Yoel Drori},
    title = {An optimal gradient method for smooth strongly convex minimization},
    journal = {Mathematical Programming},
    volume = {199},
    pages = {557-594},
    year = {2023},
}

@article{FixedPointPEP_Lieder,
    author = {Felix Lieder},
    title = {On the convergence rate of the Halpern-iteration},
    journal = {Optimization Letters},
    volume = {15},
    pages = {405-418},
    year = {2021},
}

@InProceedings{OptimalFixedPoint_Park,
  title = 	 {Exact Optimal Accelerated Complexity for Fixed-Point Iterations},
  author =       {Park, Jisun and Ryu, Ernest K},
  booktitle = 	 {Proceedings of the 39th International Conference on Machine Learning},
  pages = 	 {17420--17457},
  year = 	 {2022},
  editor = 	 {Chaudhuri, Kamalika and Jegelka, Stefanie and Song, Le and Szepesvari, Csaba and Niu, Gang and Sabato, Sivan},
  volume = 	 {162},
  series = 	 {Proceedings of Machine Learning Research},
  month = 	 {17--23 Jul},
  publisher =    {PMLR}

}

@article{AcceleratedProx_Kim,
    author = {Donghwan Kim},
    title = {Accelerated proximal point method for maximally monotone operators},
    journal = {Mathematical Programming},
    volume = {1190},
    pages = {57-87},
    year = {2021},
}

@article{TripleMomentum_Scoy,
  author={Van Scoy, Bryan and Freeman, Randy A. and Lynch, Kevin M.},
  journal={IEEE Control Systems Letters}, 
  title={The Fastest Known Globally Convergent First-Order Method for Minimizing Strongly Convex Functions}, 
  year={2018},
  volume={2},
  number={1},
  pages={49-54}
}

@article{McCormick_1976,
   title={Computability of global solutions to factorable nonconvex programs: Part {I} - Convex underestimating problems},
   volume={10},
   journal={Mathematical Programming},
   pages = {147-175},
    year = {1976},
    author = {Garth P. McCormick}
}

@article{PiecewiseMcCormick,
title = {Logic-based outer approximation for globally optimal synthesis of process networks},
journal = {Computers \& Chemical Engineering},
volume = {29},
number = {9},
pages = {1914-1933},
year = {2005},
issn = {0098-1354},
author = {Maria Lorena Bergamini and Pio Aguirre and Ignacio Grossmann},
}
    }

    \appendix
    \section{Deferred Interpolation Proofs}

\subsection{Proof of \cref{Thm:SetFuncInterp_Smooth}} \label{App:Interp_SetFuncInterp}
\edit{We first introduce notation and a key lemma}. Given $\{(x_i, g_i, f_i)\}_{i \in \I}$, define
\begin{align*}
    \conditionSmSCFunc{\mu,L}(\standardArgs)  =  f_j & - f_i - \langle g_i, x_j - x_i \rangle \\
    & - \frac{1}{2(1-\mu/L)}\left(\frac{1}{L} \|g_i - g_j \|^2 + \mu \|x_i - x_j \|^2 - \frac{2\mu}{L}\langle g_i - g_j, x_i - x_j \rangle \right) \ .
\end{align*}
Recall from \cref{Thm:func-interp} that $\conditionSmSCFunc{\mu,L}(\standardArgs) \geq 0$ for all $i,j \in \I$ if and only if $\{(x_i, g_i, f_i)\}_{i \in \I}$ is $\mathcal{F}_{\mu,L}$-interpolable.
Next, let $z_i = (x_i, f_i)$ and $v_i = (g_i, -1)$, and let $n_i = \unit{v_i}$. We further define
\begin{align*}
    \conditionSmSCSet{\mu,L}(\standardArgs) & = \frac{\mu}{2(1-\mu/L)} \left( (\om - \ol)^2 - \|z_j - \ol n_j - z_i + \om n_i \|^2 \right)  \\
    & = \frac{L-\mu}{2L\mu} - \frac{\mu}{2(1-\mu/L)} \left(\|x_j - \frac{g_j}{L\|v_j\|} - x_i +  \frac{g_i}{\mu\|v_i\|} \|^2 \right.\\
    & \qquad \qquad \qquad \qquad \qquad \qquad \qquad \left. + (f_j + \frac{1}{L\|v_j\|} - f_i - \frac{1}{\mu\|v_i\|})^2 \right) \ .
\end{align*}
From \cref{Thm:Interpolation_SuffWithSlack}, we know that $\conditionSmSCSet{\mu,L}(\standardArgs) \geq 0$ for all $i,j \in \I$ if and only if $\setInterpArgs$ is $\mathcal{C}_{\mu,L,\infty}$-interpolable \edit{(This is a direct application of \eqref{Eqn:Condition1})}.
\edit{We can relate these conditions through the following lemma, which is proven by a direct Taylor series expansion.}

\begin{lemma} \label{Lem:SetFuncConditionsEta}
    For any $\eta\geq 0$,
    \begin{align*}
        \conditionSmSCSet{\eta^2 \mu,\eta^2 L}(\rescaledArgs) = \conditionSmSCFunc{\mu,L}(\standardArgs) + \eta^2 c(x,g,f,\mu,L,\eta)
    \end{align*}
    where \edit{$\eta^2 c(x,g,f,\mu,L,\eta) \to 0$} as $\eta \to 0$.
\end{lemma}

\edit{Next we consider smoothness of the epigraph of $f$, $\epi f$. This idea is explored further in \cref{Sec:EpiSmooth} but is necessary here for our proof.}
Note for any differentiable convex function $f$, $(\nabla f(x), -1) \in N_{\mathrm{epi} f}(x,f(x))$. Then, by definition, $\epi f$ being an $L$-smooth set is equivalent to \edit{a local ball upper bound $\ballUB{x}(y; L)$ holding for all $x,y$ (See \eqref{Eqn:ballUpperBound} for explicit expression)}.
This ball upper bound $\ballUB{x}(y;L)$ \edit{is strictly larger than} the analogous quadratic bound for smooth functions. That is, for all $x,y$, we have
\begin{align*}
    \ballUB{x}(y;L) & \geq \quadraticUB{x}(y;L) := f(x) + \langle \nabla f(x), y-x\rangle + \frac{L}{2}\|x - y\|^2 \ .
\end{align*}
Now, we are ready to prove \cref{Thm:SetFuncInterp_Smooth}.

    ($\Rightarrow$) Suppose that $\{(x_i,g_i, f_i)\}_{i \in \I}$ is $\mathcal{F}_{0,L}$-interpolable. So there exists an $L$-smooth convex function $f$ such that $f(x_i) = f_i$ and $\nabla f(x_i) = g_i$. 
    Observe that for any $\eta > 0$, $\hat{f}$ defined by $\hat{f}(x) = f(\eta x)$ is an $\eta^2 L$-smooth convex function with $\hat{f}(\frac{x_i}{\eta}) = f(x_i) = f_i$ and $\nabla \hat{f}(\frac{x_i}{\eta}) = \eta \nabla f( \eta \frac{x_i}{\eta}) = \eta g_i$. Then for all $x,y$, $\hat{f}(y) \leq \quadraticUB{x}(y; \eta^2 L) \leq \ballUB{x}(y; \eta^2 L)$. As a result, $\epi \hat{f}$ is $\eta^2 L$-smooth.
    Finally,  since \edit{for any $x$}, $(\nabla \hat{f}(x), -1) \in N_{\epi \hat{f}}(x, \hat{f}(x))$, \edit{this implies that} $\epi \hat{f}$ interpolates $\setInterpRescaledArgs$.

    ($\Leftarrow$) Suppose that $\setInterpRescaledArgs$ is $\mathcal{C}_{0, \eta^2L, \infty}$-interpolable for all $\eta > 0$. By our interpolability theorems, we know that \edit{$\conditionSmSCSet{0,\eta^2 L}(\rescaledArgs) \geq 0$} for all $i,j$. Applying \cref{Lem:SetFuncConditionsEta} yields\edit{, for any $\eta > 0$,}
    \begin{align*}
        \conditionSmSCFunc{0,L}(\standardArgs)  & = \conditionSmSCSet{0,\eta^2 L}(\rescaledArgs) + \eta^2 c(x,g,f,\mu,L,\eta) \geq \eta^2 c(x,g,f,\mu,L,\eta)\ .
    \end{align*}
    Considering this as $\eta$ tends to $0$ implies $\conditionSmSCFunc{0,L}(\standardArgs) \geq 0$ from which \cref{Thm:func-interp} ensures $\{(x_i, g_i, f_i)\}_{i \in \I}$ is $\mathcal{F}_{0,L}$-interpolable. \myqed

\section{Deferred Calculations on Separating Hyperplane Algorithms} \label{App:SH-Extra}

\subsection{An SDP Reformulation of the Stopping Decision Problem} \label{App:SH-SDP-reformulation}
    We follow the approach in many previous works \cite{Interpolation,Interpolation2,Drori2019} by converting our finite-dimensional problem \eqref{Eqn:OptSH-QCQP} into an SDP. We follow similar notation to that of \cite{Drori2019} and introduce it below. \edit{First, define}
    \begin{align*}
        & \gramMatrix = [x_0|z_0|z_1|\dots|z_N|n_0|n_1|\dots|n_N|w] \in \R^{d\times(2N+4)} \\
        & G = \gramMatrix^T \gramMatrix \in \mathbb{S}_+^{2N+4}
    \end{align*}
    \edit{Next, using} the standard unit basis vectors $e_i$, define special selection vectors for extracting particular elements of our matrix $G$:
    \begin{equation*}
    \begin{array}{ll}
        \mathbf{x_0} = e_1 \in \R^{2N+4} \\[2pt]
        \mathbf{z_i} = e_{i+2} \in \R^{2N+4} & i \in [0\mathrm{\,:\,}N] \\[2pt]
        \mathbf{n_i} = e_{i+(N+3)} \in \R^{2N+4} & i \in [0\mathrm{\,:\,}N] \\[2pt]
        \mathbf{w} = e_{2N+4} \in \R^{2N+4} \\[2pt]
        \mathbf{x_{i+1}} = \mathbf{x_i} - \sum_{j=0}^i H_{i,j} \mathbf{n_j} \in \R^{2N+4} \quad & i \in [0\mathrm{\,:\,}N-1] \ .
    \end{array}
    \end{equation*}
    This construction yields the useful identities $\gramMatrix \mathbf{x_i} = x_i$, $\gramMatrix \mathbf{z_i} = z_i$, $\gramMatrix \mathbf{n_i} = n_i$, and $\gramMatrix \mathbf{w} = w$ (lastly, recall our assumption that $q = 0$). 
    Denote the symmetric outer product by $x \odot y = \half x y^T + \half y x^T$.
    Then we can directly express the standard dot product of our vectors as $\Tr \left(G \left( \mathbf{x_i} \odot \mathbf{z_j} \right) \right) = \langle x_i, z_j \rangle$
    for any $i,j \in [0\mathrm{\,:\,}N]$ (and this holds analogously for other vector combinations). Lastly, we invoke the assumption that the problem dimension $d$ satisfies $d \geq 2N+4$ to guarantee that any identified $G$ can be factorized into some $\gramMatrix \in \R^{d \times (2N+4)}$, which is common practice throughout PEP literature \cite{FirstPEP} \edit{(See \cref{Rem:DimensionBound})}.
    
    We can then write our feasibility condition \eqref{Eqn:OptSH-QCQP} as
    \begin{equation} \label{Eqn:SDP_SH}
        \exists G \in \mathbb{S}_+^{2N+4} \ \mathrm{s.t.}\ \begin{cases} 
        & G \succeq 0 \\
        & \Tr \left( G \left( (\mathbf{z_i} - \oa \mathbf{n_i} - \mathbf{z_j} + \ob \mathbf{n_j} ) \odot (\mathbf{z_i} - \oa \mathbf{n_i} - \mathbf{z_j} + \ob \mathbf{n_j} ) \right) \right) \leq \frac{1}{\gamma^2}\\
        & \Tr \left( G \left( ( \mathbf{z_i} - \oa \mathbf{n_i} - \mathbf{w}) \odot ( \mathbf{z_i} - \oa \mathbf{n_i} - \mathbf{w}) \right) \right) \leq (\og - s)^2 \\
        & \Tr \left( G \left( \mathbf{w} \odot \mathbf{w} \right) \right) \leq (\ob - \delta + s)^2 \\
        & \Tr \left( G \left( \mathbf{n_i} \odot (\mathbf{z_i} - \mathbf{x_i}) \right) \right) \leq 0\\
        & \Tr \left( G \left( \mathbf{n_i} \odot \mathbf{n_i} \right) \right) = 1 \\
        & \Tr \left( G \left( \mathbf{x_0} \odot \mathbf{x_0} \right) \right) \leq R^2 \\
        \end{cases}
    \end{equation}
    with $\gamma = (\oa - \ob)^{-1}$ and $s = \max\{0, \delta - \ob \}$.

\subsection{Proof of Theorem~\ref{Thm:SH-minimax-optimal}} \label{App:SH-minimax-optimal-proof}
    We first claim that there exists $\hat{q}$ such that $B(\hat{q}, h) \subseteq C$ and $\|\hat{q} - q \| \leq h - \delta$.
    If $\ob \leq \delta$, then $h = \delta$ and we simply let $\hat{q} = q$ and the claim is trivial. Otherwise, if $\ob > \delta$, let $z \in \argmin_{y \in \bdry C} \|y - q\|$. The optimality of $z$ ensures $z - q \in N_C(z)$ and hence, by smoothness of $C$, $B(z - \ob \unit{z-q}, \ob) \subseteq C$.
    If $\|z-q\| > \ob$, then $B(q,\ob) \subseteq C$, so we again let $\hat{q} = q$, with $h = \ob$.
    Lastly, suppose $\|z-q\| \leq \ob$. Then setting $\hat{q} = z - \ob \unit{z-q}$ yields 
    \begin{align*}
        \|q - \hat{q}\| & = \|z - \ob \unit{z-q} - q\| 
        = \ob - \| z-q\| \leq \ob - \delta = h - \delta
    \end{align*}
    where the inequality follows from $z \in \bdry C$, so $\|z-q\| \geq \delta$. Since $B(\hat{q}, \ob) \subseteq C$, this proves our first claim.
    
    \noindent {\it Proof of Claimed Stopping Time.} Suppose that by iteration $k$, the method has not yet terminated. Since $x_k \notin \interior C$, $x_k \notin \interior B(\hat{q},h)$. By definition of $n_k$ as a separating hyperplane, we know $\langle n_k, y- x_k \rangle \leq 0$ for all $y \in B(\hat{q},h)$. Considering $y = \hat{q} + hn_k$, this yields $\langle n_k, x_k - \hat{q} \rangle \geq h$.
    Using this fact, every iteration $k$ satisfies
    \begin{align*} 
        \|x_{k+1} - \hat{q}\|^2 &=  \|x_k - \hat{q}\|^2 + h^2 - 2h \langle n_k, x_k - \hat{q} \rangle \leq \|x_k - \hat{q}\|^2 - h^2 \ .
    \end{align*}
    Inductively applying this ensures that if all $x_0,\dots, x_N\not\in \interior C$, then
    \begin{equation}
    \label{Eqn:SHInduction}
        \|x_N - \hat{q} \|^2 \leq \|x_0 - \hat{q}\|^2 - N h^2 \leq (R+h-\delta)^2 - N h^2
    \end{equation}
    where we use the fact that if $\delta > \ob$, then $R + h-\delta = R \geq \|x_0 - q\| = \|x_0 - \hat{q}\|$, and if $\delta \leq \ob$, then $\|x_0 - \hat{q} \| \leq \|x_0 - q\| + \|q - \hat{q}\| \leq R + (\ob - \delta) = R + h -\delta$.
    Noting $x_N \notin \interior C$ implies that $x_N \notin \interior B(\hat{q},h)$, i.e. $\|x_N - \hat{q} \| \geq h$, it follows from $(R+h-\delta)^2 - Nh^2 \geq \|x_N - \hat{q}\|^2 \geq h^2$ that $N \leq \frac{(R+h-\delta)^2}{h^2} - 1.$
    Hence after $N = \lfloor \frac{ (R+h-\delta)^2}{h^2} \rfloor$ iterations, some iterate $x_k$ must have lied in the interior of $C$, halting the algorithm.

    \noindent {\it Proof of Matching Lower Bound.} Finally, we establish an exactly matching lower bound on any method of the form~\eqref{Eqn:SH-general-step}. For simplicity, we will assume that $N = \frac{ (R+h-\delta)^2}{h^2}$ is an integer, however, one can adjust our construction to handle non-integer cases. Below we construct a hard problem instance such that for any separating hyperplane method, $x_{k} \notin \interior C$ for all $k\leq N-1$.
    
    Let $x_0 = 0$, $q = \edit{(R/\sqrt{N}, \dots, R/\sqrt{N})} \in \R^N$, and $\hat{q} = (h, \dots, h)$. Further define $C = B(\hat{q},h) \subseteq \R^N$ and observe that $q \in \Interior{\delta} C$ and $\|x_0 - q\| = \edit{(R/\sqrt{N}) \sqrt{N}} = R$. For each $k$, we can select $n_k = -e_{k+1}$, the negative unit basis vector in $\R^N$. Note this is a valid separating hyperplane due to $x_k$'s support being its first $k$ entries as $\langle n_k, x_k \rangle = 0 \geq \langle n_k, y \rangle$ for all $y \in C$. Moreover, this choice of $n_k$ ensures the support of $x_{k+1}$ will only increase by one.
    Noting all $k\leq N-1$ have final entry equal to zero, $\|x_{k} - \hat{q}\| \geq h$ and so $x_{k} \notin \interior C$. Therefore no separating hyperplane \edit{method} can identify an interior point faster than the above simple constant stepsize method.

\section{Deferred Calculations on Frank-Wolfe Methods} \label{App:FW-Extra}

\edit{To fully justify our PEP formulation, we need to derive two simple propositions. We will once again consider separately the cases where $g_\star = 0$ and $g_\star \neq 0$ (See \cref{Subsec:FW-Derive-PEP}).} In the results below, let $\I = [0\mathrm{\,:\,}N-1]$ and $\K = [0\mathrm{\,:\,}N]$, along with $\IStar = \I \cup \{\star\}$ and $\KStar = \K \cup \{\star\}$.

\begin{proposition}\label{Thm:FW-Interpol-Ext}
    Consider any points $z_i$, $x_\star$, and $x_{i+1} = (1-h_i)x_i + h_i z_i$ \edit{and vectors $g_i$} and define
    \begin{equation*}
        \SExternal = \SSpecialExterior
    \end{equation*}
    where $g_\star \neq 0$.
    Then there exist $f \in \mathcal{F}_{\mu,L}$ and $C \in \mathcal{C}_{\alpha,\beta,D}$ satisfying $z_i \in \argmin_{y \in C} \langle \nabla f(x_i),y\rangle$ for all $i \in \I$, $x_\star \in \argmin_{y \in C} f(y)$, and $x_0 \in C$ if and only if  $\SExternal$ is $\C_{\alpha,\beta,D}$-interpolable and $\{(x_k, g_k, f_k)\}_{k\in \KStar}$ is $\mathcal{F}_{\mu,L}$-interpolable,
    where $f(x_k) = f_k$ and $\nabla f(x_k) = g_k$ for all $k \in \KStar$.
\end{proposition}

\begin{proof}
    ($\Rightarrow$) Suppose that there exist $f \in \F_{\mu,L}$ and $C \in \C_{\alpha,\beta,D}$ such that $x_\star \in \argmin_{y \in C} f(y)$ with $\nabla f(x_\star) \neq 0$ and applying Frank-Wolfe with $x_0 \in C$ yields $z_i \in \argmin_{y \in C} \langle \nabla f(x_i), y \rangle$
    for all $i \in I$. By construction, setting $g_k = \nabla f(x_k)$ and $f(x_k) = f_k$ for all $k \in \KStar$ ensures $\{(x_k, g_k, f_k)\}_{k\in \KStar}$ is $\mathcal{F}_{\mu,L}$-interpolable.    
    Since $C$ is convex, the optimality condition defining $z_i \in \argmin_{y \in C} \langle g_i, y \rangle$ implies that $z_i \in C$ and $-g_i \in N_C(z_i)$. Similarly, the optimality condition for $x_\star$ ensures that $-g_\star \in N_C(x_\star)$. \edit{Since $x_0 \in C$, we have that $\SExternal$ is $\C_{\alpha,\beta,D}$-interpolable.}

    ($\Leftarrow$) Suppose that $\SExternal$ is $\C_{\alpha,\beta,D}$-interpolable and $\{(x_k, g_k, f_k)\}_{k\in \KStar}$ is $\mathcal{F}_{\mu,L}$-interpolable. That is, there exists a $\mu$-strongly convex, $L$-smooth function $f$ such that that $f(x_k) = f_k$ and $\nabla f(x_k) = g_k$ for all $k \in \KStar$ and an $\alpha$-strongly convex, $\beta$-smooth set $C$ such that \edit{$x_0 \in C$ and $z_i \in C$ and $-g_i \in N_C(z_i)$ for all $i \in \IStar$}. Since $C$ is convex, we know that $-\nabla f(x_i) \in N_C(z_i)$ implies that $z_i \in \argmin_{y \in C} \langle \nabla f(x_i), y \rangle$. Similarly, $-\nabla f(x_\star) \in N_C(x_\star)$ implies that \edit{$x_\star \in \argmin_{y \in C} f(y)$}. Hence the desired $f$ and $C$ exist. \myqed
\end{proof}

\begin{proposition}\label{Prop:FW-Interpol-Int}
    Consider any points $z_i$, $x_\star$, and $x_{i+1} = (1-h_i)x_i + h_i z_i$ \edit{and vectors $g_i$} and define
    \begin{equation*}
        \SInternal = \SSpecialInterior \ .
    \end{equation*}
    \edit{Then} there exist $f \in \mathcal{F}_{\mu,L}$ and $C \in \mathcal{C}_{\alpha,\beta,D}$ satisfying $z_i \in \argmin_{y \in C} \langle \nabla f(x_i),y\rangle$ for all $i \in \I$ with $x_0 \in C$ and \edit{$x_\star \in \argmin_{y \in C} f(y)$} with $x_\star \in \Interior{\delta} C$ if and only if \edit{$\nabla f(x_\star) = 0$}, $\SInternal$ is $\C_{\alpha,\beta,D}$-interpolable, and $\{(x_k, g_k, f_k)\}_{k\in \KStar}$ is $\mathcal{F}_{\mu,L}$-interpolable,
    where $f(x_k) = f_k$ and $\nabla f(x_k) = g_k$ for all $k \in \KStar$.
\end{proposition}

\begin{proof}
    We apply the same approach as above, \edit{but now use the fact that $x_\star = \argmin_{y \in C} f(y)$ and $x_\star \in \Interior{\delta} C$ imply $\nabla f(x_\star) = 0$, and $\nabla f(x_\star) = 0$ implies $x_\star = \argmin_{y \in C} f(y)$}. \myqed
\end{proof}

\section{Deferred Calculations on Epismooth Gradient Methods} \label{App:ES-Extra}

\subsection{Proof of Lemma~\ref{Lem:PEP-rescaling}} \label{App:ES-PEP-rescale-proof}
Denote $b_i = \frac{\nabla f (x_i)}{\sqp{\|\nabla f(x_i)\|^2 }} \in \R^d$ and $t_i = \frac{-1}{\sqp{\|\nabla f(x_i)\|^2 }} \in \R$. Then this follows from a simple change of variables, $\tilde{x}_i = \frac{x_i}{\eta}$ and $\tilde{f}_i = \frac{f_i}{\eta}$, as

\begin{align*}
    \pEpiSmooth(L,\eta R) & = \left\{ \begin{array}{lll} \max\limits_{x_i,b_i,f_i} \quad & f_N \\
        \text{s.t.} & x_{i+1} = x_i - \ol \sum^i_{j=0} H_{i,j} b_j \quad & \forall i \in \I \\
        & \|x_0 - x_\star \|^2 \leq \eta^2 R^2 \\
        & b_\star = 0, \quad f_\star = 0, \quad x_\star = 0 \\
        & t_k = -\sqrt{1-\|b_k\|^2} \quad & \forall k \in \KStar \\
        & \langle b_k, x_\el - \frac{1}{L} b_\el - x_k + \frac{1}{L} b_k \rangle  \\
        & \qquad \quad + \langle t_k, f_\el- \frac{1}{L} t_\el - f_k + \frac{1}{L} t_k \rangle \leq 0 \quad & \forall k,\el \in \KStar
    \end{array} \right. \\
    & = \left\{ \begin{array}{lll} \max\limits_{\tilde{x}_i,b_i,\tilde{f}_i} \quad & \eta \tilde{f}_N \\
        \text{s.t.} & \tilde{x}_{i+1} = \tilde{x}_i - \frac{1}{\eta L} \sum_{j=0}^i H_{i,j} b_j \quad & \forall i \in \I \\
        & \|\tilde{x}_0 - \tilde{x}_\star \|^2 \leq R^2 \\
        & b_\star = 0, \quad \tilde{f}_\star = 0, \quad \tilde{x}_\star = 0 \\
        & t_k = -\sqrt{1-\|b_k\|^2} \quad & \forall k \in \KStar \\
        & \langle b_k, \tilde{x}_\el - \frac{1}{\eta L} b_\el -  \tilde{x}_k + \frac{1}{\eta L} b_k \rangle  \\
        & \qquad \quad + \langle t_k, \tilde{f}_\el- \frac{1}{\eta L} t_\el - \tilde{f}_k + \frac{1}{\eta L} t_k \rangle \leq 0 \quad & \forall k,\el \in \KStar
    \end{array} \right. \\
    & = \eta \, \pEpiSmooth(\eta L, R) \ .
\end{align*}

\noindent Similarly, for $p_\mathrm{S}$, the first result follows using the change of variables $\tilde{x}_i = \frac{x_i}{\eta}$, $\tilde{g}_i = \frac{g_i}{\eta}$ and $\tilde{f}_i = \frac{f_i}{\eta^2}$, and the second from $\tilde{g}_i = \frac{g_i}{\eta}$ and $\tilde{f}_i = \frac{f_i}{\eta}$.

\subsection{Derivation of SDP with Rank Constraints for Epismooth Gradient Methods} \label{App:EpiSDPDerivation}

\edit{Assume without loss of generality that $x_\star = 0$ and $f_\star = 0$.} Using $b_i$ and $t_i$ as defined in \cref{App:ES-PEP-rescale-proof}, we define
\begin{align*}
    & \gramMatrix = [x_0|b_0|\dots|b_N]  \in \R^{d \times (N+2)}\\
    & v = (f_0, \dots, f_N, t_\star,t_0,\dots,t_N) \in \R^{1\times(2N+3)}
\end{align*}
with $G = \gramMatrix^T \gramMatrix \in \mathbb{S}^{N+2}_+$ and $F = v^T v \in \mathbb{S}^{2N+3}_+$. We define $\mathbf{x_i}$ and $\mathbf{b_i}$ as selection vectors relative to $G$, similar to \cref{App:SH-Extra}, and we define $\mathbf{f_i}$ and $\mathbf{t_i}$ as selection vectors relative to $F$. We can then encode $\pEpiSmooth(L,R)$ as an SDP with an additional rank-1 constraint:
\begin{equation}    \label{Eqn:SDP_EpiSmooth_FullVersion}
    (\pEpiSmooth(L,R))^2 = \left\{\begin{array}{lll}
                \max\limits_{F,G,v} \quad & \Tr \left( F \left( \mathbf{f_N} \odot \mathbf{f_N} \right) \right) \\
                \text{s.t. } \qquad \quad & G \succeq 0, \qquad  F \succeq 0 \\
                & \Tr \left( G \left( (\mathbf{x_0} - \mathbf{x_\star}) \odot (\mathbf{x_0} - \mathbf{x_\star}) \right) \right) \leq R^2 \\
                & \Tr \left( G \left( \mathbf{b_k} \odot \mathbf{b_k} \right) + F \left( \mathbf{t_k} \odot \mathbf{t_k} \right) \right) = 1 \quad & \forall k \in \KStar\\
                & \Tr \left( G \left( \mathbf{b_k} \odot (\mathbf{x_\el} - \ol \mathbf{b_\el} - \mathbf{x_k} + \ol \mathbf{b_k}) \right) \right. \\
                    & \left. \qquad  + F \left( \mathbf{t_k} \odot(\mathbf{f_\el} - \ol \mathbf{t_\el} - \mathbf{f_k} + \ol \mathbf{t_k}) \right) \right) \leq 0 \quad & \forall k,l \in \KStar \\
                & \Tr \left( F(\mathbf{t_\star} \odot \mathbf{t_\star}) \right) = 1 \\
                & \Tr \left( F(\mathbf{t_\star} \odot \mathbf{f_0} ) \right) \leq 0 \\
                & \Tr \left( F(\mathbf{t_k} \odot \mathbf{f_0} )\right) \leq 0 \quad & \forall k \in \K \\
                & F = v^T v \ .
            \end{array} \right.
\end{equation}

In order to enforce $t_\star = -1$, we use the two constraints $t_\star^2 = 1$ and $t_\star f_0 \leq 0$. The first ensures that $t_\star = \pm 1$. By convexity of $f$, we know that $f_0 \geq f_\star = 0$. If $f_0 = 0$, then $b_0 = 0$, and consequently $f_N = 0$, so this is irrelevant for the worst-case instance. So we know effectively that $f_0 > 0$ and conclude that $t_\star = -1$. Similarly, to enforce that $t_k \leq 0$, we require $t_k f_0 \leq 0$.

\subsection{Proof of Theorem~\ref{Thm:Epi-Smooth-PEP-Convergence}} \label{App:ES-Proof-Epi-PEP-Convergence}

From \cref{Lem:PEP-rescaling}, we know that $\frac{1}{\eta^2} p_\mathrm{S}(L,\eta R) = p_\mathrm{S}(L,R)$ for all $\eta$. So the second equality is immediate. It remains to show $\lim_{\eta \to 0} \frac{p_\mathrm{ES}(L, \eta R)}{\eta^2} = p_\mathrm{S}(L,R)$.

Let $\rho_\eta := (\etaIdx{x}, \etaIdx{g}, \etaIdx{f})$ denote the sequence of \edit{maximizers} of $\pEpiSmooth(\eta^2 L, \frac{R}{\eta})$ as $\eta \to 0$. We consider the rescaled sequence $s_\eta = ( \eta \etaIdx{x}, \frac{\etaIdx{g}}{\eta}, \etaIdx{f})$.
By our method being eventually-epismooth-stable and the fact that $\rho_\eta$ is a solution to $\pEpiSmooth(\eta^2L, \frac{R}{\eta})$, we must have $\|\etaIdx{x}_i\| \leq C \frac{R}{\eta}$, $\|\etaIdx{g}_i\| \leq C \eta^2 L \frac{R}{\eta}$, and $|\etaIdx{f}_i| \leq C \eta^2 L \frac{R^2}{\eta^2} $ for all $i$. 
Consequently, we have $\|\eta \etaIdx{x}_i\| \leq C R$, $\| \frac{\etaIdx{g}_i}{\eta} \| \leq C LR$, and $|\etaIdx{f}_i| \leq C L R^2$ for all $i$. This shows that $s_\eta$ belongs to a compact set for all $\eta$. We consider $\limsup_{\eta \to 0} \pEpiSmooth(\eta L, \frac{R}{\eta})$ and define $\eta_k$ as the subsequence attaining the $\limsup$. By compactness, consider a further subsequence $\eta_{k'}$ such that $s_{\eta_{k'}}$ converges to some limit point $s^*$.

Next, we claim that $s^*$ is a feasible point for $p_\mathrm{S}(L,R)$. We know that each $\rho_\eta = ( \etaIdx{x}, \etaIdx{g}, \etaIdx{f} )$ satisfies the constraints
\begin{equation*}
    \begin{cases}
        & \etaIdx{x}_{i+1} = \etaIdx{x}_i - \frac{1}{\eta^2 L} \sum_{j=0}^i H_{i,j} \frac{\etaIdx{g}_j}{\sqp{\|\etaIdx{g}_j\|^2}} \\
        & \|\etaIdx{x}_0 - \etaIdx{x}_\star \|^2 \leq \frac{R^2}{\eta^2} \\
        & \conditionSmSCSet{0,\eta^2 L}(\etaIdx{x},\etaIdx{g},\etaIdx{f}) \geq 0 \ .
    \end{cases}
\end{equation*}
Then rescaling with $\etaIdx{\tilde{x}} = \eta \etaIdx{x}$, $\etaIdx{\tilde{g}} = \frac{\etaIdx{g}}{\eta}$, and $\etaIdx{\tilde{f}} = \etaIdx{f}$ yields
\begin{equation*}
    \begin{cases}
        & \etaIdx{\tilde{x}}_{i+1} = \etaIdx{\tilde{x}}_i - \frac{1}{L} \sum_{j=0}^i H_{i,j} \frac{\etaIdx{\tilde{g}}_j}{\sqp{\eta^2\|\etaIdx{\tilde{g}}_j\|^2}} \\
        & \|\etaIdx{\tilde{x}}_0 - \etaIdx{\tilde{x}}_\star \|^2 \leq R^2 \\
        & \conditionSmSCSet{0,\eta^2 L}(\frac{\etaIdx{\tilde{x}}}{\eta},\eta \etaIdx{\tilde{g}},\etaIdx{\tilde{f}}) \geq 0 \ .
    \end{cases}
\end{equation*}
\edit{From \cref{Lem:SetFuncConditionsEta}, taking the limit as $\eta \to 0$, we have that $\lim_{\eta \to 0} \conditionSmSCSet{\eta^2 \mu, \eta^2 L}(\rescaledArgs) \geq 0$ if and only if $\conditionSmSCFunc{\mu,L}(\standardArgs) \geq 0$.
Applying this and} the continuity of our constraints, $s^* = (\tilde{x}^*, \tilde{g}^*, \tilde{f}^* )$ must satisfy
\begin{equation*}
    \begin{cases}
        & \tilde{x}_{i+1}^* = \tilde{x}_i^* - \frac{1}{L} \sum_{j=0}^i H_{i,j} \tilde{g}_j^* \\
        & \|\tilde{x}_0^* - \tilde{x}_\star^* \|^2 \leq R^2 \\
        & \conditionSmSCFunc{0,L}(\tilde{x}^*,\tilde{g}^*,\tilde{f}^*) \geq 0 \ .
    \end{cases}
\end{equation*}
Therefore $s^*$ is a feasible solution to $p_\mathrm{S}(L,R)$ and so $p_\mathrm{S}(L,R) \geq \tilde{f}^*_N-\tilde{f}^*_\star$.

We know from \edit{the proof of \cref{Lem:SetFuncConditionsEta}} that for all $\eta$, $\pEpiSmooth(\eta^2 L, \frac{R}{\eta}) \geq p_\mathrm{S}(\eta^2 L, \frac{R}{\eta})$. Combining this with our rescaling result in \cref{Lem:PEP-rescaling}, we have
\begin{equation*}
    \liminf_{\eta \to 0} \pEpiSmooth(\eta^2 L, \frac{R}{\eta}) \geq \liminf_{\eta \to 0} p_\mathrm{S}(\eta^2 L, \frac{R}{\eta}) = p_\mathrm{S}(L,R) \ .
\end{equation*}
Combining our results above, we can squeeze the limit as follows
\begin{align*}
\liminf_{\eta \to 0} \pEpiSmooth(\eta^2 L, \frac{R}{\eta}) & \geq p_\mathrm{S}(L,R) \geq \tilde{f}^*_N-\tilde{f}^*_\star 
 = \lim_{k \to \infty} f^{(\eta_k')}_N-f^{(\eta_k')}_\star = \limsup_{\eta \to 0} \pEpiSmooth(\eta^2L, \frac{R}{\eta}) \ . 
\end{align*}
Therefore, $\lim_{\eta \to 0} \pEpiSmooth(\eta^2 L, \frac{R}{\eta}) = p_\mathrm{S}(L,R)$.
Applying \cref{Lem:PEP-rescaling} twice gives the final claim.

\end{document}